\documentclass[a4paper,12pt]{article}

\usepackage{mathpazo}
\usepackage{url}
\usepackage{amsmath,amssymb, amsthm}
\usepackage{stmaryrd}
\usepackage{graphicx}
\usepackage{textcomp}
\usepackage{enumitem}
\usepackage{xspace}
\usepackage{bigints}

\usepackage[utf8]{inputenc}
\usepackage{geometry}
\usepackage[affil-it]{authblk}

\newcommand{\normL}[2]{\|#1\|_{L^2_r(#2)}}

\newcommand{\normsob}[2]{\|#1\|_{H^k(#2)}}
\newcommand{\Xnorm}[1]{\| #1\|_{X^k}}

\newtheorem{theorem}{Theorem}[section]
\newtheorem{definition}[theorem]{Definition}
\newtheorem{lemma}[theorem]{Lemma}

\newtheorem{proposition}[theorem]{Proposition}
\newtheorem{remark}[theorem]{Remark}

\newtheorem{notation}[theorem]{Notation}
\title{ \bf Floating structures in shallow water: local well-posedness in the axisymmetric case}
\author{\small EDOARDO BOCCHI\footnote{Institut de Mathématiques de Bordeaux UMR 5251, Université de Bordeaux,  351 Cours de la Libération, 33405 Talence, France -  \texttt{edoardo.bocchi@math.u-bordeaux.fr}}}
\date{}
\begin{document}
\maketitle 
\begin{abstract}
The floating structure problem describes the interaction between surface water waves and a floating body, generally a boat or a wave energy converter. 
As shown by Lannes in \cite{Lan}, the equations for the fluid motion can be reduced to a set of two evolution equations on the surface elevation and the horizontal discharge. The presence of the object is accounted for by a constraint on the discharge under the object; the pressure exerted by the fluid on this object is then the Lagrange multiplier associated with this constraint. Our goal in this paper is to prove the well-posedness of this fluid-structure interaction problem in the shallow water approximation under the assumption that the flow is axisymmetric without swirl. We write the fluid equations as a quasilinear hyperbolic mixed initial boundary value problem and the solid equation as a second order ODE coupled to the fluid equations. Finally we prove the local in time well-posedness for this coupled problem, provided some compatibility conditions on the initial data are satisfied.
\end{abstract}

\section{Introduction}
The floating structure problem was first formulated by John in his two famous papers \cite{John,John2}, in which he considered a linear flow. It is a particular example of fluid-structure interaction, when a partially immersed body is floating on the fluid free surface. The fluid is supposed to be incompressible and the flow irrotational, and we can consider both the case of a prescribed solid motion or of a free motion governed by Newton's law for the solid object. In this problem we have to treat two free boundary problems.\\ The first free boundary problem is the typical water waves problem, which consists in describing the evolution of the fluid surface in contact with the air (or another fluid whose density can be neglected). In the absence of floating bodies, this is the standard water waves problem, which has been studied by many authors in the last years and whose local well-posedness theory is well-known. For instance we refer to Wu \cite{Wu2D,Wu3D}, Lannes \cite{LanWel,LanWWP} , Alazard, Burq and Zuily \cite{AlBuZu2011,AlBuZu2014} and Iguchi \cite{Iguchi}. A notable formulation is the one introduced by Zakharov, Craig and Sulem
\cite{Zakharov, CraigSulem}: they consider the potential velocity and they remove the dependence on the vertical variable $z$ working with new unknowns, the free surface elevation and the trace of the potential on the free surface.\\ 
The second free boundary problem is given by the fact that the portion of the body in contact with the fluid depends on time, so that the contact line between is a free boundary problem. For this difficulty John studied a more simplified problem. He considered a linear model in order to describe the evolution of the free surface waves and he used the potential velocity formulation. Then, he assumed that the motion of the solid is of small amplitude and he neglected the variations of the contact line in time. These assumptions permitted him to avoid the free boundary problem associated with the contact line. Moreover he studied a 
one-dimensional problem (where $d$ is the horizontal dimension). Another way to avoid this free boundary problem is to consider a structure with vertical side-walls and to assume its motion to be only vertical.\\
Even if John's approach is simplified because it does not take into account nonlinear effects, the linear approach has been used extensively in hydrodynamic engineering. In particular we refer to Cummins who, dealing with ship motion, proposed in \cite{cummins1962} his celebrated delay differential equation on the six modes of response: surge, sway, heave, roll, pitch and yaw.\\
In order to take into account nonlinear effects to better describe the real motion of floating bodies, Lannes proposed a different approach in his recent paper \cite{Lan}. He modelled the problem not with the velocity potential theory but using a new formulation. In order to remove the dependence on the vertical variable $z$, he considered the horizontal discharge, \text{i.e.} the horizontal component of the velocity field integrated vertically between the free surface and the fluid domain bottom. He showed that the equations of the problem have a \textquotedblleft compressible-incompressible\textquotedblright  structure, in which the interior pressure exerted by the fluid on the body is a Lagrange multiplier that one can determine via the resolution of a $d$-dimensional elliptic equation. He also implemented the same approach on asymptotic models, such as the nonlinear shallow water equations and Boussinesq equations.\\\\
In this paper we address the two-dimensional floating body problem, where a cylindrically symmetric structure with vertical side-walls is floating only vertically on an incompressible fluid with irrotational motion. These assumptions on the shape of the solid and its motion permit us to avoid the free boundary problem associated with the contact line and to simplify the problem. Indeed in this case the projection of the portion of the body in contact with the fluid does not depend on time. We suppose that the flow is axisymmetric and without swirl, \text{i.e.} we consider a rotation-invariant velocity field with no azimuthal component. Moreover we consider the shallow water regime, which means that the wavelength of the waves is larger than the depth. Consequently, we work with the nonlinear shallow water equations for the flow model, instead of the much more involved free surface Euler equations. The one-dimensional case with any assumption on the solid motion and shape has been studied by Iguchi and Lannes in \cite{IguLan}, where they proposed a general approach to study the well-posedness of initial boundary value problems with a free boundary.\\
The aim of this paper is to prove the local in time well-posedness of this coupled fluid-structure problem in Sobolev spaces. We need enough regular initial data, satisfying some compatibility conditions, in order to get the solution. We consider here a two-dimensional problem, but the axisymmetry keeps the boundary condition maximally dissipative. We use this condition to get better trace estimates which in general we do not have in a two-dimensional case. With all these assumptions we can reduce the problem to a one-dimensional radial problem, and we adapt the classical theory with a reformulation for weighted spaces. The important point is that, instead of John's model, in this paper we take into account nonlinear terms. On the other hand our model still has some limitations.
For instance, one should allow the solid to move also in the horizontal direction and to rotate. In these cases, one cannot bypass the study of the evolution of the contact line; moreover the flow would cease to be axisymmetric. The axisymmetric situation considered here is however relevant and can be used to validate the shallow water approach to the floating body problem: indeed, several experimental data with an axisymmetric geometry are available.  

\subsection{Outline of the paper}
In Section 2 we write the free surface Euler equations with the constraint that the solid must be in contact with the fluid during all the motion, avoiding air holes between them. Then, we write the nonlinear shallow water approximation for this floating structure problem using the same formulation as Lannes \cite{Lan}, with the introduction of the horizontal discharge. Once we have assumed the axisymmetry and the absence of swirl, introducing cylindrical coordinates we reformulate the problem to get a one-dimensional set of equations.\\In Section 3 we focus on the \textquotedblleft fluid part\textquotedblright  of the problem. We write the floating structure problem in the exterior domain $(R,+\infty)$ as a quasilinear hyperbolic initial boundary value problem, namely 
\begin{equation}
\begin{cases}
\partial_t h_e +\partial_r q_e +\dfrac{1}{r}q_e=0\\[10pt]
\partial_t q_e + \partial_r \left(\dfrac{q_e^2}{h_e}\right)+\dfrac{q_e^2}{rh_e}+gh_e\partial_r h_e =0\\
\end{cases} \mbox{in} \ \ \ (R,+\infty)
\label{systemIntro}
\end{equation} 
coupled with the boundary condition
\begin{equation}
	{q_e}_{|_{r=R}}=q_i{_{|_{r=R}}}.
	\label{boundaryintro}
\end{equation}
where $h_e=\zeta_e +h_0$ and $\zeta_e$ are the fluid height and the free surface elevation in the exterior domain respectively (the flat bottom is parametrized by $-h_0$). The exterior horizontal discharge $q_e$ is defined as
$$q_e(t,r)=\int_{-h_0}^{\zeta_e}u_r(t,r,z)dz \ \ \  \ \ \ \ \mbox{in} \ \ \ \ \ \ (R,+\infty),$$where $u_r$ is the radial component of the fluid velocity field $\mathbf{U}$.  In the boundary condition \eqref{boundaryintro} $q_i$ is the interior horizontal discharge, defined as before but for in the interior domain $(0,R).$\\
As usual we first show the well-posedness of the associated linear problem in $L^2(rdr)$. In order to have more regular solutions, the data of the problem must satisfy some compatibility conditions. Then, we apply a standard iterative scheme argument to get the existence and uniqueness of the linear solution.\\
 We treat the solid motion in Section 4. We write the vertical component of Newton's law for the conservation of linear momentum, which describes the motion of the structure, as a nonlinear second order ODE on the displacement of the vertical position of the solid center of mass from its equilibrium position, denoted by $\delta_G$. We show that this ODE can be written under the form 
 \begin{equation}
 (m+m_a(\delta_G)) \ddot{\delta}_G(t)= -\mathfrak{c} \delta_G(t) + \mathfrak{c} \zeta_e (t,R) + \left(\frac{\mathfrak{b}}{h_e^2(t,R)}+\beta(\delta_G)\right)\dot{\delta}_G^2(t).
 \label{eqmotionIntro}
 \end{equation} 
 In this equation the terms $\zeta_e(t,R)$ and $h_e(t,R)$ are responsible for the coupling with the fluid equations \eqref{systemIntro}.
  An important point is the presence of the added mass term $m_a(\delta_G)$ which can be explained by the fact that, in order to move, the solid has to accelerate itself and also the portion of fluid around it. We can find the added mass effect also in other fluid-structure interaction problems. For instance, in the case of a totally submerged solid we refer to Glass, Sueur and Takahashi \cite{SueurTaka} and Glass, Munnier and Sueur \cite{VortexP}. Moreover, this effect has an important role for the stability of numerical simulations \cite{NumAddesMass}.\\
  The value of the elevation of the fluid surface at the boundary $\zeta_e(t,R)$ in \eqref{eqmotionIntro} is the coupling term with the fluid motion. On the other hand we show that the boundary condition in \eqref{systemIntro} can be written as
  $${q_e}_{|_{r=R}}=q_i{_{|_{r=R}}}=-\frac{R}{2}\dot{\delta}_G,$$ showing the retro-action of the solid on the fluid motion.\\
Finally in Section 5 we write the coupled system modelling the problem and we show the local in time existence and uniqueness introducing an iterative scheme on the coupled fluid-structure system, looking for the solution via a fixed point argument, and we get the following result (see Theorem \ref{maintheorem} for the rigorous statement):
\begin{theorem}
	The coupled system \eqref{systemIntro} - \eqref{eqmotionIntro} is local in time well-posed, provided the initial data are regular enough and verify some compatibility conditions. 
\end{theorem}
 In Appendix A we show the details in the case of a non-flat solid bottom, considering that the contact between the solid and the fluid is still on the vertical side-walls, and we derive the corresponding solid motion ODE. In Appendix B we show the proof of a product estimate for functions in $H^k((0,T))$.

 \subsection*{Acknowledgements}
 The author addresses his sincere gratitude to D. Lannes for the supervision of this paper and his precious advices. The author wants to acknowledge C. Prange for his useful remarks and T. Iguchi for some interesting conversations. He also wants to thank M. Tucsnak, who noticed that, in the configuration we consider in this paper, a correction pressure term must be introduced to have exact energy conservation.\\\\ The author was supported by the French National Research Agency project NABUCO, grant ANR-17-CE40-0025, the Fondation Simone et Cino Del Duca and the Conseil Régional Nouvelle Aquitaine.
\section{Floating structure equations }
\subsection{Constrained free surface Euler equations}
Let us consider a floating body, typically a wave energy converter, with vertical side-walls and a cylindrical symmetry, forced to move only in the vertical direction. We call $C(t)$ the region occupied by the solid at time $t$, $\partial C(t)$ the boundary and $\partial_wC(t)$ the portion of the boundary in contact with the fluid, called the \textit{wetted surface}. The presence of the solid naturally allows to divide the horizontal plane $\mathbb{R}^2$ into two regions, the projection $\mathcal{I}$, of the wetted surface on it, and $\mathcal{E}:=\mathbb{R}^2\setminus\overline{\mathcal{I}}$. We call them \textit{interior} and \textit{exterior domain} respectively. The boundary $\Gamma:=\partial\mathcal{I}=\partial\mathcal{E}$ is called the projection of the \textit{contact line}, where the solid, the fluid and the exterior air interact. For simplicity we call $\Gamma$ itself  the contact line. These domains do not depend on time since the solid is moving only vertically and is assumed to have vertical side-walls. We consider a wetted surface that can be parametrized as graph of some function $\zeta_w(t,X)$ for $X\in \mathcal{I}$ and, like in the water waves theory, we assume that the surface of the fluid is the graph of a function $\zeta(t,X)$ for $X\in\mathbb{R}^2$, as shown in Figure \ref{image1}.\\\\
We assume that the fluid is incompressible, irrotational, with constant density $\rho$ and inviscid. For simplicity we consider a flat bottom which can be parametrized by $-h_0$ with $h_0>0$ and the fluid domain is $$\Omega(t)=\{(X,z)\in \mathbb{R}^{2+1} | -h_0 <z<\zeta(t,X) \}.$$Then, the motion of the fluid is given by the incompressible Euler equation
\begin{align}
\partial_t \mathbf{U}+ \mathbf{U}\cdot\nabla_{X,z}\mathbf{U}&=-\dfrac{1}{\rho}\nabla_{X,z}P -g\mathbf{e_z} \ \ \mbox{ in } \ \Omega(t) \label{euler}\\ \mbox{div }\mathbf{U}&=0\\\mbox{curl }\mathbf{U}&=0.
\end{align}
 The boundary conditions for the Euler equation on the velocity field $\mathbf{U}$ in the fluid domain are the traditional kinematic equation at the surface and the impermeability condition at the bottom, respectively
 \begin{align} \label{kinematic}\centering 
 z=\zeta, \ \ \ \	\partial_t \zeta -\mathbf{U}\cdot N&=0  \ \ \ \mbox{with  } N=\binom{-\nabla \zeta}{1}
\\ z=-h_0, \ \ \  \ \ \ \ \ \ \mathbf{U}\cdot \mathbf{e}_z&=0
 \end{align} 

	\begin{figure}
		\centering
		\includegraphics[scale=1]{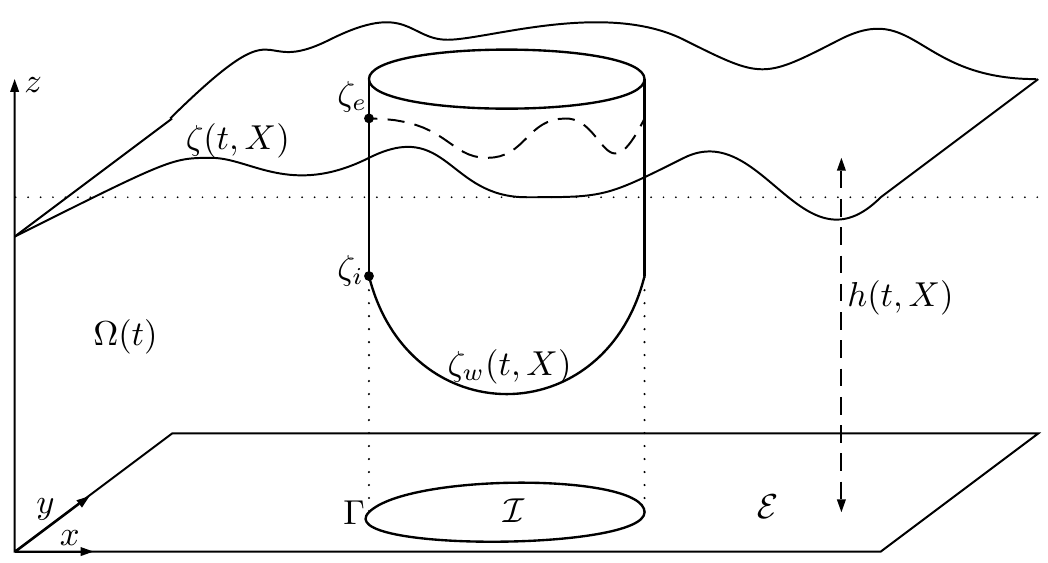}
		\caption{A cylindrically symmetric floating structure with vertical side-walls}
		\label{image1}
	\end{figure}
	
We consider a configuration when the fluid is completely attached to the solid. Hence we have the following contact constraint:
$$\zeta(t,X)=\zeta_w(t,X) \  \ \mbox{in}\ \  \mathcal{I}.$$ 
Let us denote the restrictions to the interior domain and the exterior domain of a function $f$ defined on $\mathbb{R}^2$ as $$f_i:=f_{|_{\mathcal{I}}} \ \ \ \ \ \ \ \ \ \ f_e:=f_{|_{\mathcal{E}}}.$$

According to this notation the contact constraint becomes the following \begin{equation}
	\zeta_i=\zeta_w \ \ \ \mbox{in} \ \ \mathcal{I}
	\label{interiorconstraint}
\end{equation}
\begin{remark}
	Since the solid has vertical side-walls, the free surface is not continuous across $\Gamma$, \textit{i.e.} $$\zeta_e\neq\zeta_i \ \ \ \ \mbox{on} \ \ \Gamma.$$
	\label{rem1}
\end{remark}
In the presence of a floating structure we have to change the standard condition on the value of the pressure on the free surface. In the exterior domain it is given by the constant atmospheric pressure $P_{\mathrm{atm}}$, $i.e.$ \begin{equation}
	\underline{P}_e=P_{\mathrm{atm}}.
	\label{pressurecondition}
\end{equation} with $\underline{P}=P_{|_{z=\zeta}}$. In the interior domain the pressure on the free surface is an unknown of the problem, depending on the dynamics of the solid 
but we know its value on $\Gamma$. Indeed, by integrating the vertical component of Euler's equation \eqref{euler} between $z=\zeta_i$ and $z=\zeta_e$, we have
\begin{equation}
\underline{P}_i(t,\cdot)= P_{\mathrm{atm}} +\rho g (\zeta_e-\zeta_i)+ \rho\int_{\zeta_i}^{\zeta_e}(\partial_t w + \mathbf{U}\cdot\nabla_{X,z}w)  \ \ \ \mbox{on} \ \ \ \Gamma,
\label{discontinuityelevation}
\end{equation}where $w$ is the vertical component of the velocity field $\mathbf{U}$. The second and the third term do not vanish due to the discontinuity of the free surface on $\Gamma$ (see Remark \ref{rem1}). \\
Moreover one has the continuity of the normal velocity at the vertical side-walls, $i.e.$
\begin{equation}
	V\cdot \nu = V_C\cdot \nu
	\label{normalvelocity}
\end{equation} where $\nu$ is the unit normal vector to $\Gamma$ pointing towards $\mathcal{E}$, $V$ and $V_C$ are the horizontal velocities of the fluid and the solid respectively.\\As in the standard water waves theory we suppose that the height of the fluid $h_e(t)$ does not vanish during all the motion. Hence we have the following assumption:
\begin{equation}
\exists \  h_{m}>0 : \ \ h_e(t,X)\geq h_{m} \ \ \ \forall t\in [0,T), \forall X\in\mathcal{E}.
\label{he>hm}
\end{equation}From the physics of the floating structure problem we suppose also that the solid does not touch the bottom of the domain during its motion. This is equivalent to assuming that the height of the fluid $h_i(t,X)$ under the solid does not vanish, \text{i.e.}
\begin{equation}
\exists \  h_{min}>0 : \ \ h_w(t,X)\geq h_{min} \ \ \ \forall t\in [0,T), \forall X\in\mathcal{I}.
\label{hw>hmin}
\end{equation}with $h_w(t,X)=h_i(t,X)$ in $\mathcal{I}$ due to \eqref{interiorconstraint}. This assumption is completely relevant for the situation investigated here; we refer to \cite{dePoyferre, Ming} (Euler equation) and \cite{LanShoreline} (nonlinear shallow water and Green-Naghdi equations) for the analysis of the vanishing depth problem.
\subsection{Averaged  free surface Euler equations with a floating structure}
Because of the difficulty to deal with a moving domain, we want to obtain a set of evolution equations on $\mathbb{R}^2$, without the dependence on the vertical variable $z$. This idea has been used also for the water waves problem without a floating structure by Zakharov-Craig-Sulem (see \cite{Zakharov}, \cite{CraigSulem} and \cite{LanWWP}), who introduce the trace of the velocity potential on the free surface. 
In the presence of a floating structure we use another formulation, the one that Lannes implemented in his paper \cite{Lan}. We will see that this formulation permits us to write directly the boundary condition \eqref{normalvelocity}. We define the horizontal discharge  $$Q(t,X):=\int_{-h_0}^{\zeta(t,X)}V(t,X,z)dz=h(t,X)\overline{V}(t,X)$$ where $h(t,X)=h_0+\zeta(t,X)$ is the fluid height and $\overline{V}$ is the vertical average of $V$. Then, we can reformulate the problem:

\begin{proposition}
	Using the $(\zeta,Q)$-formulation, the water waves problem with a floating structure is modelled by the following system 
	\begin{equation}
	\begin{cases}
	\partial_t \zeta + \nabla\cdot Q=0,\\[10pt]
	\partial_t Q + \nabla\cdot\left(\dfrac{1}{h}Q\otimes Q\right) + gh\nabla\zeta + \nabla\cdot \mathbf{R}+ h \mathbf{a}_{NH}(h,Q) 
	= -\dfrac{h}{\rho}\nabla\underline{P},
	\end{cases}
	\label{WWFloat}
	\end{equation}with the \textquotedblleft Reynolds\textquotedblright  tensor $\mathbf{R}$ and the non-hydrostatic acceleration $\mathbf{a}_{NH}$ as 
	\begin{equation*}
	\mathbf{R}(h,Q)=\int_{-h_0}^{\zeta} \left(V-\overline{V}\right)\otimes\left(V-\overline{V}\right),
	\end{equation*}
	\begin{equation*}
	\mathbf{a}_{NH}(h,Q)=\frac{1}{h}\int_{-h_0}^{\zeta}\nabla \left[\int_{z}^{\zeta}\left(\partial_t \mathbf{U} + \mathbf{U}\cdot \nabla_{X,z}\mathbf{U}\right)\cdot \mathbf{e}_z \right].
	\end{equation*}
	The surface pressure $\underline{P}$ is given by 
	\begin{equation}
	\underline{P}_e=P_{\mathrm{atm}}\ \ \ \mbox{and}\ \ \ \begin{cases}
	-\nabla\cdot \left(\dfrac{h}{\rho}\nabla \underline{P}_i\right)= -\partial_t^2\zeta_w + \mathbf{a}_{FS}(h, Q) \  \ \ \ \ \mbox{in} \ \ \mathcal{I}\\[10pt]
	{\underline{P}_i}_{|_{\Gamma(t)}}= P_{\mathrm{atm}}+ \rho g (\zeta_e-\zeta_i) + P_{NH} \ \ \ \mbox{on} \ \ \ \Gamma,
	\end{cases}
	\label{pres}
	\end{equation}where $$\mathbf{a}_{FS}(h, Q)=\nabla \cdot \left( \nabla\cdot\left(\dfrac{1}{h}Q\otimes Q\right)+ gh\nabla\zeta+\nabla\cdot \mathbf{R}+ h \mathbf{a}_{NH}(h,Q) \right),$$
	\begin{equation*}
		P_{NH}=\rho \int_{\zeta_i}^{\zeta_e}\left(\partial_t \mathbf{U} + \mathbf{U}\cdot \nabla_{X,z}\mathbf{U}\right)\cdot \mathbf{e}_z 
	\end{equation*} and the transition condition at the contact line is \begin{equation}
	Q_e\cdot \nu=Q_i\cdot \nu \ \ \ \mbox{on} \ \ \ \Gamma.
	\label{transition}
	\end{equation}
\end{proposition}

\begin{remark}
	Note that the expressions for $\mathbf{R}$ and $\mathbf{a}_{NH}$ in the statement above involve the velocity field $\mathbf{U}=\binom{V}{w}$. It is shown in \cite{Lan} that the velocity field is fully determined by the knowledge of $\zeta$ and $Q$, hence the notations $\mathbf{R}(h,Q)$ and $\mathbf{a}_{NH}(h,Q).$ 
\end{remark}

\begin{proof}[Sketch of the proof]
	In order to get \eqref{WWFloat} we just need to integrate over the vertical coordinate the horizontal component of the Euler equation \eqref{euler}. The condition on the surface in the exterior domain is \eqref{pressurecondition}. One can note that, in the interior domain, $\underline{P}_i$ is the Lagrange multiplier associated with the constraint \eqref{interiorconstraint}.Then, using the standard argument for incompressible Euler, the elliptic system in \eqref{pres} is obtained by applying the horizontal divergence to \eqref{WWFloat} in the interior domain. The boundary condition for the elliptic system comes from \eqref{discontinuityelevation}. Finally for \eqref{transition} we use the definition of $Q_e$ and $Q_i$ and \eqref{normalvelocity} (see Proposition 10 in \cite{Lan}).
\end{proof}

\begin{remark}\label{energy}
	The fluid energy is
	\begin{equation*}
		E_{\mbox{fluid}}= \frac{\rho}{2}g\int_{\mathbb{R}^2}\zeta^2 +\frac{\rho}{2}\int_{\Omega(t)}|\mathbf{U}|^2.
	\end{equation*}
\end{remark}
\subsection{The shallow water regime}
We are interested here in a simplified model of these equations in the shallow water regime, when the horizontal scale of the problem is much larger than the depth. In the water waves problem the horizontal scale is given by the typical wavelength of the waves.
\begin{remark}
	In the floating structure problem a third relevant length in the floating structures problem is the solid width. In particular the ratio of the solid width and the wavelength naturally appears in the adimensionalized equations. Here we consider this quantity as a parameter independent of the other parameters, such as the shallowness and the nonlinear parameters, and it takes no role in the derivation of the asymptotic model. 
\end{remark}	
Following \cite{Lan}, the same approximations as in the case without floating object are made, namely
\begin{equation*}
\mathbf{R}\approx 0 \ \  \ \ \ \ \mbox{and} \ \ \ \ \  	\mathbf{a}_{NH}\approx 0. 
\end{equation*}
Then, we get the following equations:
\begin{proposition}
The nonlinear shallow water equations with a floating structure in the $(\zeta,Q)$-formulation are 
	\begin{equation}
	\begin{cases}
	\partial_t \zeta + \nabla\cdot Q=0,\\[10pt]
	\partial_t Q + \nabla\cdot\left(\dfrac{1}{h}Q\otimes Q\right) + gh\nabla\zeta 
	= -\dfrac{h}{\rho}\nabla\underline{P}.
	\end{cases}
	\label{SWFloat}
	\end{equation} with the surface pressure $\underline{P}$ given by 
	\begin{equation}
	\underline{P}_e=P_{\mathrm{atm}}\ \ \ \mbox{and}\ \ \ \begin{cases}
	-\nabla\cdot \left(\dfrac{h}{\rho}\nabla \underline{P}_i\right)= -\partial_t^2\zeta_w + \mathbf{a}_{FS}(h, Q) \ \ \ \mbox{in} \ \ \mathcal{I}\\[10pt]
	{\underline{P}_i}_{|_{\Gamma}}= P_{\mathrm{atm}}+ \rho g (\zeta_e-\zeta_i)_{|_{\Gamma}} + P_{\mathrm{cor}} ,
	\end{cases}
	\label{SWpres}
	\end{equation}where $$\mathbf{a}_{FS}(h, Q)=\nabla \cdot \left( \nabla\cdot\left(\dfrac{1}{h}Q\otimes Q\right)+ gh\nabla\zeta\right),$$
	coupled with the transition condition at the contact line \begin{equation}
	Q_e\cdot \nu=Q_i\cdot \nu \ \ \ \mbox{on} \ \ \ \Gamma.
	\label{SWtransition}
	\end{equation}
\end{proposition}
Differently from the case considered by Lannes in \cite{Lan}, where the jump of pressure $\underline{P}_{i_{|_{\Gamma(t)}}}-P_{\mathrm{atm}}$ at the boundary of the object is assumed to be only due to the hydrostatic pressure, \textit{i.e.}$$\underline{P}_{i_{|_{\Gamma(t)}}}-P_{\mathrm{atm}}=\rho g\left(\zeta_e-\zeta_i\right)_{|_{\Gamma(t)}},$$ we add here a non-hydrostatic correction term $P_{\mathrm{cor}}$. This corrector is determined later in Proposition \ref{conservationenergy} below to ensure exact energy conservation (the mathematical analysis if we remove this term can be performed in the same way).
~\\
As for the kinematic condition \eqref{kinematic}, we have that 
\begin{equation}
	\partial_t \zeta_w -\underline{U}_w \cdot N_w  \ \ \ \mbox{in} \ \ \mathcal{I} \  \ \ \mbox{with} \ \ \ \ N_w=\binom{-\nabla \zeta_w}{1}
	\label{kinematicsolid}
\end{equation}where $\underline{U}_w$ is the velocity of the solid on the wetted surface. Let us denote the center of mass of the solid $G(t)=(X_G(t), z_G(t))$ and $\mathbf{U}_G(t)=(V_G(t),w_G(t))$ its velocity and $\omega$ the angular velocity of the solid. From the solid mechanics we have
\begin{equation*}
	\underline{U}_w= \mathbf{U}_G + \omega \times \mathbf{r}_G \ \  \ \ \ \mbox{with} \ \ \ \ \ \mathbf{r}_G(t,X)=\binom{X-X_G(t)}{\zeta_w(t,X)-z_G(t)}.
\end{equation*} Then, \eqref{kinematicsolid} gives
\begin{equation}
	\partial_t\zeta_w =\left( \mathbf{U}_G + \omega \times \mathbf{r}_G\right)\cdot N_w \ \ \ \  \mbox{in} \ \ \ \mathcal{I}.
	\label{timederzetaw}
\end{equation}
Because of the linearity of the elliptic problem we can decompose the interior pressure as $\underline{P}_i=\underline{P}_i^\mathrm{I}+\underline{P}_i^\mathrm{II}+\underline{P}_i^\mathrm{III}$ where:
	 \begin{itemize}
	 	\item $\underline{P}_i^\mathrm{I}$ is the pressure we would have in the case of a fixed solid, solution to 
	\begin{equation} \begin{cases}
	-\nabla\cdot \left(\dfrac{h}{\rho}\nabla \underline{P}_i^\mathrm{I}\right)= \mathbf{a}_{FS}(h, Q) \ \ \ \mbox{in} \ \ \mathcal{I},\\[10pt]
	{\underline{P}_i^{\mathrm{I}}}_{|_{\Gamma}}= P_{\mathrm{atm}},
	\end{cases}
	\label{Pr1}\end{equation}
	where $\mathbf{a}_{FS}(h, Q)$ is the free surface acceleration in the absence of a floating structure;\\
		\item $\underline{P}_i^\mathrm{II}$ is the part of the pressure due to the acceleration of the solid 
\begin{equation}
	\begin{cases}
	-\nabla\cdot \left(\dfrac{h}{\rho}\nabla \underline{P}_i^\mathrm{II}\right)= -\partial_t^2 \zeta_w,\ \ \ \mbox{in} \ \ \mathcal{I}, \\[10pt]
	{\underline{P}_i^{\mathrm{II}}}_{|_{\Gamma}}= 0,
	\end{cases}
	\label{Pr2}
	\end{equation}
	where $w_G$ is the vertical component of the velocity of the center of mass $G(t)$ of the solid;\\
	\item $\underline{P}_i^\mathrm{III}$ is the part of the pressure due to the pressure discontinuity at the contact line
\begin{equation}
	\begin{cases}
	-\nabla\cdot \left(\dfrac{h}{\rho}\nabla \underline{P}_i^\mathrm{III}\right)= 0 \ \ \ \mbox{in} \ \ \mathcal{I},\\[10pt]
	{\underline{P}_i^{\mathrm{III}}}_{|_{\Gamma}}= \rho g (\zeta_e - \zeta_i)_{|_{\Gamma}} + P_{\mathrm{cor}}.
	\end{cases}
	\label{Pr3}
	\end{equation}
	\end{itemize}
	 
\subsection{Axisymmetric without swirl setting}
\begin{figure}
	\centering
	\includegraphics[scale=1]{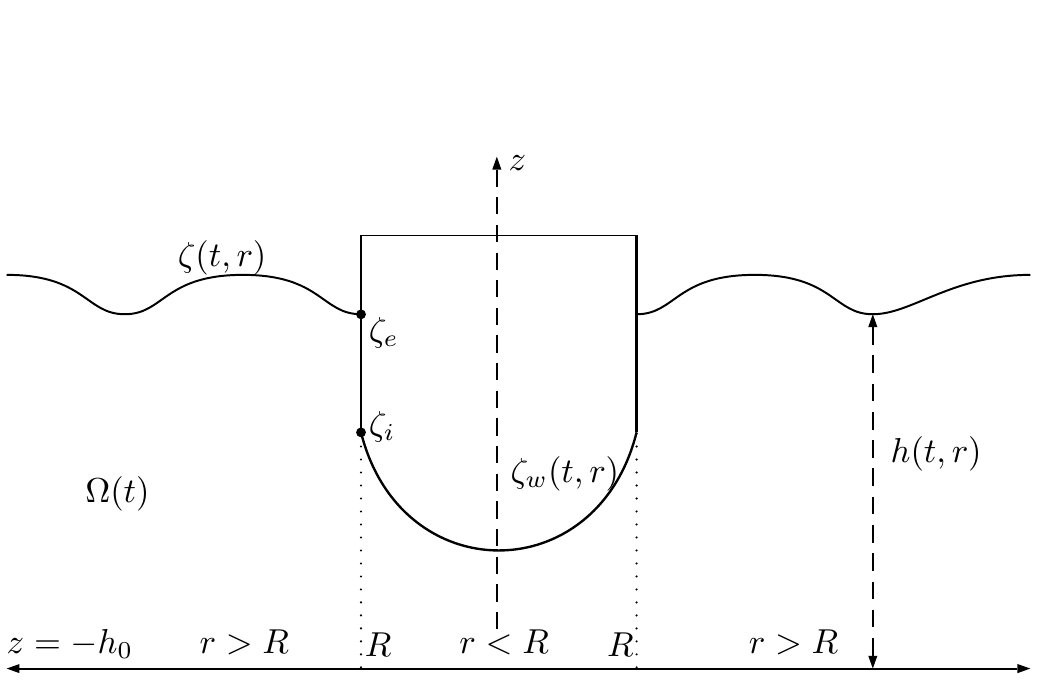}
	\caption{Vertical cross section of the axisymmetric configuration}
	\label{image2}
\end{figure}

Without loss of generality we suppose the center of mass to have coordinates $G(t)=(0,0, z_G(t))$ and let $R$ be the radius of the interior domain $\mathcal{I}$. Introducing a cylindrical coordinates system with the $z$-axis coincident with the axis of symmetry of the solid (see Figure \ref{image2}) we write the velocity field $\mathbf{U}$ as $$\mathbf{U}(t,r,\theta,z)=\left(u_r(t,r,\theta,z), u_\theta(t,r,\theta,z), u_z(t,r,\theta,z)\right).$$ From now on and throughout the paper we consider an axisymmetric flow without swirl, which means that the flow has no dependence on the angular variable $\theta$, \text{i.e.} $\mathbf{U}=\mathbf{U}(t,r,z)$, and $u_\theta=0$ respectively. Hence the horizontal discharge can be written as
\begin{equation*}
	Q(t,r)=\left(q_r(t,r), 0 \right)
\end{equation*}with 
\begin{equation*}
	q_r(t,r)=\int_{-h_0}^{\zeta}u_r(t,r,z)dz 
\end{equation*}and the tangential component vanishes since 
$$q_\theta(t,r)=\int_{-h_0}^{\zeta}u_\theta(t,r,z)dz=0.$$ For simplicity we write $q$ instead of $q_r$ for the radial component of the horizontal discharge. Moreover, since the solid moves only vertically and the swirl is neglected in the flow, $V_G=0$ and $\omega=0$. Hence from \eqref{timederzetaw} we have 
\begin{equation*}
\partial_t^2 \zeta_w = \dot{w}_G.
\end{equation*} In the new system of reference the shallow water model \eqref{WWFloat} - \eqref{transition} becomes 
\begin{equation}
\begin{cases}
\partial_t h +\partial_r q +\dfrac{q}{r}=0\\[10pt]
\partial_t q + \partial_r \left(\dfrac{q^2}{h}\right)+\dfrac{q^2}{rh}+gh\partial_r h =-\dfrac{h}{\rho}\partial_r\underline{P}
\end{cases}  \mbox{in} \ \ \ (0,+\infty)
\label{WWFloatpolar}
\end{equation} coupled with the transition condition
\begin{equation}
q_{e_{|_{r=R}}} =q_{i_{|_{r=R}}}.
\label{neumannq}
\end{equation}
We have $\underline{P}_e=P_{\mathrm{atm}}$ and \eqref{Pr1} - \eqref{Pr3} become
	\begin{gather} \begin{cases}
	-\left(\partial_r + \dfrac{1}{r}\right)\left(\dfrac{h_w}{\rho}\partial_r \underline{P}_i^\mathrm{I}\right)= \\[10pt] \left(\partial_r + \dfrac{1}{r}\right)\left(\partial_r \left(\dfrac{q_i^2}{h_w}\right)+\dfrac{q_i^2}{rh_w}+gh_w\partial_r h_w\right) \ \ \ \mbox{in} \ \ (0,R)\\[10pt]
	{\underline{P}_i^{\mathrm{I}}}_{|_{r=R}}= P_{\mathrm{atm}},
	\label{P1}
	\end{cases}\end{gather}
	\begin{gather}
	\begin{cases}
	-\left(\partial_r + \dfrac{1}{r}\right)\left(\dfrac{h_w}{\rho}\partial_r \underline{P}_i^\mathrm{II}\right)= -\dot{w_G} \  \ \ \ \ \mbox{in} \ \ (0,R)\\[10pt]
	{\underline{P}_i^{\mathrm{II}}}_{|_{r=R}}= 0,
	\end{cases}
	\label{P2}
	\end{gather}
	\begin{gather}
	\begin{cases}
	- \left(\partial_r + \dfrac{1}{r}\right)\left(\dfrac{h_w}{\rho}\partial_r \underline{P}_i^\mathrm{III}\right)= 0 \ \ \ \ \ \ \ \mbox{in} \ \ (0,R)\\[10pt]
	{\underline{P}_i^{\mathrm{III}}}_{|_{r=R}}= \rho g (\zeta_e - \zeta_i)_{|_{r=R}} + P_{\mathrm{cor}},
	\end{cases}
	\label{P4}
	\end{gather}
	where we replace $h_i=\zeta_i + h_0$ with $h_w=\zeta_w+h_0$ due to the contact constraint \eqref{interiorconstraint}. Using axisymmetry and absence of swirl together with this change of coordinates we pass from a two-dimensional to a one-dimensional problem, where explicit calculations can be done (see Section 4). With these assumptions, the horizontal discharge is no more a vectorial quantity but a scalar quantity, making the problem easier to handle. 

\begin{remark}
	Under the shallow water approximation and in the axisymmetric without swirl setting the fluid energy in Remark \ref{energy} becomes
	\begin{equation}
	E_{SW}= 2\pi\frac{\rho}{2}g\int_{0}^{+\infty}\zeta^2rdr +2\pi\frac{\rho}{2}\int_{0}^{+\infty}\frac{q^2}{h}rdr.
	\end{equation} 		
\end{remark} 
In the presence of a floating structure  the fluid energy $E_{SW}$ is no more conserved by the equations in \eqref{SWFloat}. Let us define the energy for a solid moving only vertically as
\begin{equation}
E_{sol}=\frac{1}{2}mw_G^2+mgz_G.
\end{equation}and the total fluid-structure energy
\begin{equation}
	E_{tot}:= E_{SW}+ E_{sol}.
\end{equation}
\begin{notation}
	$$\left\llbracket f\right\rrbracket:= f_{e_{|_{r=R}}} -f_{i_{|_{r=R}}}$$ is the jump between the exterior and the interior domain at the contact line $r=R$.
\end{notation}
 We can now state the following proposition:
\begin{proposition}\label{conservationenergy}
Choosing $P_{\mathrm{cor}}=\dfrac{\rho}{2}q^2_{i_{|_{r=R}}}\left\llbracket \dfrac{1}{h^2}\right \rrbracket,$ the total fluid-structure energy is conserved, \textit{i.e.}
\begin{equation}
	\frac{d}{dt}E_{tot}= 0.
	\label{totalconser}
\end{equation}
\end{proposition}
\begin{proof}
Multiplying the first equation in \eqref{WWFloatpolar} by $\rho g \zeta r$ and the second by $\dfrac{qr}{h}$ and summing up, we use the fact that $\partial_t h =\partial_t \zeta$ to write the system under the conservative form
\begin{equation}
	\partial_t \mathfrak{e} + \partial_r F=-rq\partial_r\underline{P},
	\label{conservativeform}
\end{equation}
where $\mathfrak{e}$ is the local fluid energy $$\mathfrak{e}=\frac{\rho}{2}g\zeta^2 r + \frac{\rho}{2}\frac{q^2}{h}r$$ and $F$ is the flux
$$F=\rho\left(\frac{q^3}{2h^2}r + g\zeta q r \right).$$ The conservative form \eqref{conservativeform} reads in the interior domain $(0,R)$
\begin{equation}
	\partial_t \mathfrak{e}_i + \partial_r F_i=-rq_i\partial_r\underline{P}_i
	\label{interiorconservative}
\end{equation} and in the exterior domain $(R,+\infty)$
\begin{equation}
\partial_t \mathfrak{e}_e + \partial_r F_e=0
\label{exteriorconservative}
\end{equation}
We integrate \eqref{interiorconservative} on $(0,R)$ and \eqref{exteriorconservative} on $(R,+\infty)$ and multiplying by $2\pi$ we obtain 
\begin{equation}
	\frac{d}{dt} E_{SW} - 2\pi\rho R \left\llbracket \frac{q^3}{2h^2} + g\zeta q \right\rrbracket= -2\pi\int_{0}^{R}rq_i\partial_r\left(\underline{P}_i -P_{\mathrm{atm}}\right)dr, 
\end{equation} By integration by parts we get 
\begin{equation*}\begin{aligned}
		\frac{d}{dt} E_{SW} = & \ 2\pi\rho R \left\llbracket \frac{q^3}{2h^2} + g\zeta q \right\rrbracket -2\pi R \left(\underline{P}_i -P_{\mathrm{atm}}\right)_{|_{r=R}}q_{i_{|_{r=R}}}\\&+2\pi\int_{0}^{R}\left(\underline{P}_i -P_{\mathrm{atm}}\right)\partial_r(rq_i)dr.
		\end{aligned}
\end{equation*}
On the other hand, from the definition of $E_{sol}$, we have 
\begin{equation*}
\begin{aligned}
\frac{d}{dt}E_{sol}&= mw_G\dot{w}_G + mgw_G= w_G\left(m\dot{w}_G + mg\right)\\&=w_G \ 2\pi\int_{0}^{R}\left(\underline{P}_i -P_{\mathrm{atm}}\right)rdr\\&=2\pi\int_{0}^{R}\left(\underline{P}_i -P_{\mathrm{atm}}\right)\partial_t\zeta_wrdr
\end{aligned}
\end{equation*} where we have used Newton's law for the conservation of the linear momentum in the axisymmetric configuration
\begin{equation*}
	m\dot{w}_G = -mg + 2\pi \int_0^R \left(\underline{P}_i -P_{\mathrm{atm}}\right)rdr
\end{equation*}
and \eqref{timederzetaw}. From the contact constraint \eqref{interiorconstraint} and the  mass conservation equation in \eqref{WWFloatpolar} the following yields:
\begin{equation}
	\frac{d}{dt}E_{sol}=-2\pi\int_{0}^{R}\left(\underline{P}_i -P_{\mathrm{atm}}\right)\partial_r(rq_i )dr.
\end{equation}Therefore 
$$\frac{d}{dt}E_{sol}= -\frac{d}{dt}E_{SW} + 2\pi\rho R \left\llbracket \frac{q^3}{2h^2} + g\zeta q \right\rrbracket -2\pi R\left(\underline{P}_i -P_{\mathrm{atm}}\right)_{|_{r=R}}q_{i_{|_{r=R}}}.$$
Using the expression of the interior pressure $\underline{P}_i$ on the boundary $r=R$ in \eqref{P1} - \eqref{P4} and the transition condition \eqref{neumannq} we obtain
\begin{equation*}
	\frac{d}{dt}\left(E_{SW}+E_{sol}\right)= 2\pi R\left(\frac{\rho}{2}q^3_{i_{|_{r=R}}}\left\llbracket \frac{1}{h^2}\right \rrbracket - q_{i_{|_{r=R}}}P_{\mathrm{cor}}\right).
\end{equation*}If $q_{i_{|_{r=R}}}=0$ the result follows directly. Otherwise, we choose the pressure corrector term $P_{\mathrm{cor}}$ in \eqref{P4} as
\begin{equation}
	P_{\mathrm{cor}}= \frac{\rho}{2}q^2_{i_{|_{r=R}}}\left\llbracket \frac{1}{h^2}\right \rrbracket,
\label{Pcor}
\end{equation} and we get \eqref{totalconser}.
\end{proof}

\section{The fluid equations}
In this section we focus on the \textquotedblleft fluid part \textquotedblright of the coupled problem. We show that the exterior part of \eqref{WWFloatpolar} can be written as a one-dimensional quasilinear hyperbolic initial boundary value problem in an exterior domain and we shall prove the local in time well-posedness. Like frequently in the literature, throughout this paper we also use the term \textit{mixed} problem: this comes from the fact that we have, as data of the problem, both the initial (in time) and the boundary (in space) values.\\ Hyperbolic problems in exterior domains have been treated in many works. M\'{e}tivier \cite{Met}, Benzoni and Serre \cite{Benz} have studied hyperbolic initial boundary value problems in exterior domains with constant coefficients and maximally dissipative boundary condition. Isozaki \cite{Isozaki} and Alazard \cite{Alazard} have studied the singular incompressible limit for the compressible Euler equation in an exterior domain. Concerning the quasilinear hyperbolic mixed problems, Schochet \cite{Schochet} has proved the local in time existence in the case of bounded domains and Shibata and Kikuchi \cite{Shibata} have showed the local in time existence for some second order problem in bounded and unbounded domains. Differentiability of solutions to hyperbolic mixed problems has also been studied by Rauch and Massey \cite{Rauch}.\\The case we are considering here has not been treated in the literature yet. We consider a two-dimensional problem, but the axisymmetry keeps the boundary condition maximally dissipative which in general for a two-dimensional problem is not true. This property is essential for the coupling with the solid motion as it provides us better trace estimates than the ones for the general two-dimensional shallow water equations, but in other cases it is not necessary (see \cite{Ma-Os} for elementary examples). Hence we reduce the problem to a one-dimensional radial problem, then we must adapt the classical theory.\\
Let us recall that in the exterior domain we have
\begin{align}
&\begin{cases}
\partial_t h_e +\partial_r q_e +\dfrac{q_e}{r}=0\\[10pt]
\partial_t q_e + \partial_r \left(\dfrac{q_e^2}{h_e}\right)+\dfrac{q_e^2}{rh_e}+gh_e\partial_r h_e =0.
\end{cases} \mbox{in} \ \ \ (R,+\infty)
\label{WWFloatpolarext}
\end{align}
coupled with the boundary condition \begin{equation}{q_e}_{|_{r=R}}=q_i{_{|_{r=R}}} \label{boundarycond}\end{equation}
Defining $u=(\zeta_e,q_e)^T$ and adding the Cauchy data we can write \eqref{WWFloatpolarext} - \eqref{boundarycond} as the following quasilinear hyperbolic mixed problem
\begin{equation}
\begin{cases}
\partial_t u +A(u)\partial_r u +B(u,r)u=0 \ \ \  \ \ \ \mbox{in} \ \ (R,+\infty)\\
\mathbf{e}_2 \cdot u_{|_{r=R}}=q_i{_{|_{r=R}}}\\
u(0)=u_0
\end{cases}
\label{quasilinearU}
\end{equation} with \begin{equation*}A(u)=\left (
\begin{array}{cc}
0& 1 \\
gh_e-\dfrac{q_e^2}{h_e^2}& \dfrac{2q_e}{h_e}\\
\end{array}
\right ), \ \ \ \  B(u,r)=\left (
\begin{array}{cc}
0& \dfrac{1}{r} \\[10pt]
0& \dfrac{q_e}{rh_e}\\
\end{array}
\right )\end{equation*} and $$u_0=(\zeta_{e,0},q_{e,0})^T.$$

\subsection{The linear hyperbolic mixed problem}
In order to construct the solution to the floating structure problem, which is a quasilinear mixed problem of the form \eqref{quasilinearU} coupled with Newton's equation for the solid motion, we shall use an iterative scheme based on the following linearization of \eqref{quasilinearU},
\begin{equation}
\begin{cases}
L(\overline{u})u=\partial_t u +A(\overline{u})\partial_r u +B(\overline{u},r)u=f,\\
\mathbf{e}_2 \cdot u_{|_{r=R}}=g,\\
u(0)=u_0,
\end{cases}
\label{linearU}
\end{equation}with some $\overline{u}$ and $g=q_{i_{|_{r=R}}}$. Since the coefficients of $A$ and $B$ are rational fractions, we have $A(\cdot), B(\cdot,r)\in C^\infty(\mathcal{U})$ for some open set $\mathcal{U}\in \mathbb{R}^2$ where $\overline{h_e}$ does not vanish, which represents a phase space of $\overline{u}$.\\
%
			\noindent Let us assume that we are in the subsonic regime, which means that for  $\overline{u}=(\overline{\zeta_e},\overline{q_e})$ the following holds:
			\begin{equation}
				\dfrac{\overline{q_e}^2}{\overline{h_e}^2}< g\overline{h_e}\label{subsonic}
			\end{equation}with $\overline{h_e}=h_0 +\overline{\zeta_e}$. In the case of water waves in oceans, where the water depth is much bigger than the water velocity, this assumption is satisfied. Then, for the linear initial boundary problem \eqref{linearU} we have the following result:
			\begin{proposition}\label{p3p4}
				Assume that $\overline{u}$ satisfies \eqref{subsonic}. Then, the linear exterior hyperbolic mixed problem \eqref{linearU}, which is the linearization of the floating structure equations \eqref{quasilinearU}, satisfies the following properties 
				:\\\\
			(P3) $A(\overline{u})$ has one strictly positive eigenvalue $\lambda_+(\overline{u})$ and one strictly negative eigenvalue $\lambda_-(\overline{u}),$ 
			\\\\
			(P4) $P_-(\overline{u})\mathbf{e}_2^\perp \neq (0,0)$ except in $(0,0)$,  where $P_-(\overline{u})$ is the projector on the eigenspace associated with the negative eigenvalue of $A(\overline{u})$ and $\mathbf{e}_2^\perp$ is the orthogonal complement of $\mathbf{e}_2.$
			\end{proposition}
			\begin{proof}
				The eigenvalues of $A(\overline{u})$ are $\lambda_\pm(\overline{u})=\pm\sqrt{g\overline{h_e}} + \dfrac{\overline{q_e}}{\overline{h_e}} $ and the associated unit eigenvectors are  $$\mathbf{e}_\pm(\overline{u})=\frac{1}{\sqrt{1+ \lambda_\pm ^2(\overline{u})}}(1, \lambda_\pm (\overline{u})).$$ The assumption \eqref{subsonic} gives property (P3).
				We prove now property (P4). Let us denote $P_+(\overline{u})$ and $P_-(\overline{u})$ the projectors on the eigenspaces associated with $\lambda_+(\overline{u})$ and $\lambda_-(\overline{u})$ respectively. They are given explicitly by
				\begin{equation}
				P_+(\overline{u})= \dfrac{A(\overline{u}) - \lambda_-(\overline{u})\rm{Id}}{\lambda_+(\overline{u})- \lambda_-(\overline{u})} \ \ \ \ \ \ \ P_-(\overline{u})= -\dfrac{A(\overline{u}) -\lambda_+(\overline{u})\rm{Id}}{\lambda_+(\overline{u})- \lambda_-(\overline{u})} 
				\label{projectors}
				\end{equation}  Since $\mathbf{e}_2^\perp$ is of the form $a\mathbf{e}_1$ with $a\in\mathbb{R}$, from the definition of $P_-(\overline{u})$ in \eqref{projectors} we have
				\begin{equation}
				P_-(\overline{u})\mathbf{e}_2^\perp = -\frac{a}{\lambda_+(\overline{u})- \lambda_-(\overline{u})}(\lambda_+(\overline{u}),- \lambda_+(\overline{u})\lambda_-(\overline{u}))^T,  \ \ \ \ a \in \mathbb{R}
				\end{equation} which is different from zero, except for $a=0$, since $\lambda_+(\overline{u})\neq 0$.
			\end{proof}
 The following lemma is a direct consequence of Proposition \ref{p3p4}:
\begin{lemma}\label{p1p2}
Assume that $\overline{u}$ satisfies \eqref{subsonic}. The linear hyperbolic exterior mixed problem \eqref{linearU} satisfies the following properties:
\begin{enumerate}[label=(P\arabic*)]
\item The system is \textit{Friedrichs symmetrizable}, \textit{i.e.} there exists a symmetric matrix $S(\overline{u})$, called the symmetrizer, such that there exist $\alpha>0$ such that
$S(\overline{u})\geq \alpha\text{Id} $ and $S(\overline{u})A(\overline{u})$ is symmetric.
\item The boundary condition is \textit{maximally dissipative}: $S(\overline{u})A(\overline{u})$ is negative definite on the kernel of the boundary condition $\mathbf{e}_2^\perp,$ where $\mathbf{e}_2^\perp$ is the orthogonal complement of $\mathbf{e}_2.$.
\end{enumerate}
\end{lemma}
\begin{proof}
 From $(P3)$ we have that $\lambda_+(\overline{u})>0$ and $\lambda_-(\overline{u})<0$. We define the symmetrizer $S(\overline{u}):=M P^T_-(\overline{u})P_-(\overline{u}) + P^T_+(\overline{u})P_+(\overline{u})$ for some constant $M>0$. We compute that
	\begin{align*}
	(S(\overline{u})v,v)&=((MP_-^T(\overline{u})P_-(\overline{u})+ P_+^T(\overline{u})P_+(\overline{u}))v,v)\\&=M (P_-(\overline{u})v, P_-(\overline{u})v) + (P_+(\overline{u})v, P_+(\overline{u})v). 
	\end{align*}
	Hence , from the decomposition $v= P_+(\overline{u})v + P_-(\overline{u})$ we get
	\begin{equation}
	(S(\overline{u})v,v)\geq \alpha(v,v)
	\end{equation}
	with $\alpha=\min(M,1)/2$. The symmetry of $S(\overline{u})$ is trivial. We have the following spectral decomposition 
	\begin{equation}
	A(\overline{u})= \lambda_+(\overline{u})P_+(\overline{u}) + \lambda_-(\overline{u}) P_-(\overline{u}).
	\end{equation}
	By the definition of the projectors \eqref{projectors},  $S(\overline{u})A(\overline{u})$ reads
	\begin{equation}
		S(\overline{u})A(\overline{u})= \lambda_- M P_-^T(\overline{u})P_-(\overline{u}) + \lambda_+ P_+^T(\overline{u}) P_+(\overline{u})
			\end{equation}which is clearly symmetric and we get property (P1). We refer to Taylor (see Prop. 2.2 of \cite{Taylor}) for a different proof with a general notion of symmetrizer involving pseudo-differential operators.\\ 
	Let us consider $\mathbf{e}_2^\perp$, the one dimensional orthogonal complement of $\mathbf{e}_2\in\mathbb{R}^2$, which is the kernel of the boundary condition. Then, we compute that
	\begin{align*}
	(S(\overline{u})A(\overline{u})\mathbf{e}_2^\perp,\mathbf{e}_2^\perp)
	=& \ \lambda_-(u)M(P_-(\overline{u})\mathbf{e}_2^\perp, P_-(\overline{u})\mathbf{e}_2^\perp)\\& + \lambda_+(\overline{u})(P_+(\overline{u})\mathbf{e}_2^\perp, P_+(\overline{u})\mathbf{e}_2^\perp). 
	\end{align*}
	Due to property (P4) we obtain property (P2) choosing  $$M>-\dfrac{\lambda_+(\overline{u}) (P_+(\overline{u})\mathbf{e}_2^\perp,P_+(\overline{u})\mathbf{e}_2^\perp)}{\lambda_-(\overline{u})(P_-(\overline{u})\mathbf{e}_2^\perp,P_-(\overline{u})\mathbf{e}_2^\perp)}.$$
\end{proof}
	\begin{remark}
		Property (P4) is a reformulation of the uniform Kreiss-Lopatinski\u{i} condition. Then, we have just proved that the system \eqref{linearU} admits a Kreiss symmetrizer, which transforms the system into a symmetric one with the additional property that the boundary condition for this symmetric system is maximally dissipative. This property will permit to control the trace of the solution at t he boundary by the standard energy estimate. 
	\end{remark} Let us now introduce the following space:
\begin{equation}X^k(T):=\bigcap_{j=0}^k C^j([0,T], H_r^{k-j}((R,+\infty)))$$ endowed with the norm $$\| u\|_{X^k(T)}:=\sup_{t\in[0,T]}\Xnorm{u(t)} \ \ \ , \ \ \  \Xnorm{u(t)}=\sum_{j=0}^{k}\|\partial_t^j u(t)\|_{H_r^{k-j}((R,+\infty))}
	\label{Xk}
\end{equation} with $\| \cdot\|_{H_r^k((R,+\infty))}$ the norm of the weighted Sobolev space $H^k((R,+\infty),rdr)$. 
We first show the following a priori estimate useful to find strong solutions of the problem \eqref{linearU}.

\begin{proposition}\label{prioriestL2}
	Let $T>0$ and $\overline{u}\in X^2(T)$ be such that \eqref{subsonic} is satisfied. With $\alpha$ as in Lemma \ref{p1p2} there are constants $c_{\alpha,R}$ and a non-decreasing function $C_R(\cdot)$ on $[0,+\infty)$ such that all the solutions $u \in H^1_r((0,T)\times (R,+\infty)) $ solving \eqref{linearU} satisfy 
\begin{equation}\begin{aligned}
\normL{u(t)}{(R,+\infty)}^2& + \|u_{|_{r=R}}\|_{L^2((0,t))}^2 \leq c_{\alpha,R}e^{tC_{\alpha,R}(\overline{u})} \times \\& \times \left( \normL{u_0}{(R,+\infty)}^2+\|{g}\|_{L^2((0,t))}^2+ \int_{0}^{t}\normL{f(\tau)}{(R,+\infty)}^2d\tau\right)
\label{aprioriestimate}
\end{aligned}\end{equation}for all $t\in [0,T],$ with $C_{\alpha,R}(\overline{u})=
1+\alpha^{-1}C_R(\|\overline{u}\|_{X^2(T)})$.
\end{proposition}
\begin{proof}
Following Proposition 2.2 of \cite{Benz}, we have from property (P1) and by integrations by parts 
\begin{equation*}
\begin{aligned}
&\dfrac{d}{dt}(S(\overline{u})u,u)_{L^2_r((R,+\infty))}=-2(S(\overline{u})A(\overline{u})\partial_r u,u)_{L^2_r((R,+\infty))}\\&-2(S(\overline{u})B(\overline{u},r) u,u)_{L^2_r((R,+\infty))}+((\partial_t S(\overline{u})) u,u)_{L^2_r((R,+\infty))}+ 2(S(\overline{u})f,u)_{L^2_r((R,+\infty))}\\
&=S(\overline{u})A(\overline{u})u_{|_{r=R}}\cdot u_{|_{r=R}}R  + \left(W(\overline{u}) u,u\right)_{L^2_r((R,+\infty))} +\left(S(\overline{u}) u,f\right)_{L^2_r((R,+\infty))}.
\end{aligned}
\end{equation*}
with \begin{equation}W(\overline{u})=\partial_tS(\overline{u}) +\partial_r (S(\overline{u})A(\overline{u})) +\frac{1}{r}S(\overline{u})A(\overline{u})-2S(\overline{u})B(\overline{u},r)\label{W}\end{equation}
Property (P2) permits us to control the first term on the right-hand side of the inequality, using the following Lemma from Métivier \cite{Met}:
\begin{lemma}
The symmetric matrix $S(\overline{u})A(\overline{u})$ is negative definite on $\mathbf{e}_2^\perp$, the set of all vectors orthogonal to $\mathbf{e}_2$, if and only if there are constants $c_1,c_2>0$ such that for each vector $h\in \mathbb{C}^2$:
$$-(S(\overline{u})A(\overline{u})h,h)\geq c_1|h|^2 - c_2|\mathbf{e}_2\cdot h|^2.$$
\end{lemma}
Choosing $h=u_{|_{r=R}}$, integrating in time and using property (H1) we have \begin{equation}
\begin{aligned}
&\left(S(\overline{u}) u(t),u(t)\right)_{L^2((R,+\infty))} \leq \left(S(\overline{u}(0)) u_0,u_0\right)_{L^2((R,+\infty))}\\&+\int_{0}^{t}\left(S(\overline{u}) f(\tau),f(\tau)\right)_{L^2((R,+\infty))}+c_2R\left|{g(\tau)}\right|^2-c_1R\left|u_{|_{r=R}(\tau)}\right|^2d\tau\\&+(1+\alpha^{-1}\|W(\overline{u})\|_{L^\infty((0,T)\times(R, +\infty) )})\int_{0}^{t}\left(S(\overline{u})u(\tau),u(\tau)\right)_{L^2((R,+\infty))}d\tau.
\end{aligned}\label{ineqtang}\end{equation} 
where we have used the fact that $S(\overline{u})\geq \alpha\text{Id}.$ Moreover, we get the following estimate:
\begin{equation}
\begin{aligned}
	&\|W(\overline{u})\|_{L^\infty((0,T)\times(R, +\infty) )} \\& = \|S^\prime(\overline{u})\partial_t\overline{u} + S^\prime(\overline{u})\partial_r\overline{u}A(\overline{u}) +S(\overline{u})A^\prime(\overline{u})\partial_r\overline{u}\\&\ \ \ \ \ \ \ + \frac{1}{r} S(\overline{u})A(\overline{u}) - 2S(\overline{u})B(\overline{u},r)\|_{L^\infty((0,T)\times(R, +\infty) )}  \\ 
	& \leq C_R(\|\overline{u}\|_{W^{1,\infty}((0,T)\times(R, +\infty) )})\leq C_R\left(\|\overline{u}\|_{X^2(T)}\right)
\end{aligned}
\label{normW}
\end{equation}for some non-decreasing function $C_R(\cdot)$ on $[0,+\infty)$, where we have used the fact that $r\in(R,+\infty)$, the embedding $H^2_r((R,+\infty))\hookrightarrow W^{1,\infty}((R, +\infty))$ and the definition of $X^2(T).$
By Gronwall's Lemma we have
\begin{align*}
\alpha\|{u(t)}\|_{X^0}^2 +& c_1R\|u_{|_{r=R}}\|_{L^2((0,t))}^2 \leq  \ e^{t\left(1+\alpha^{-1}C_R(\|\overline{u}\|_{X^2(T)})\right)}\times \\&\times \left( c(\|\overline{u}(0)\|_{X^2})\|{u_0}\|_{X^0}^2+c_2R\|{g}\|_{L^2((0,t))}^2+ \alpha^{-1}\int_{0}^{t}\|{f(\tau)}\|^2_{X^0}d\tau\right).
\end{align*}
We get \eqref{aprioriestimate} for $$C_{\alpha,R}(\overline{u})=
1+\alpha^{-1}C_R(\|\overline{u}\|_{X^2(T)}), \ \ \ \ \ \ \ \ c_{\alpha,R}=\dfrac{\max(c(\|\overline{u}(0)\|_{X^2}),c_2 R)}{\min(\alpha,c_1R)}.$$
\end{proof}
Then, following Theorem 2.4.5 of \cite{Met}, one can show that there is a unique solution $u\in C^0\left([0,T], L^2_r((R,+\infty))\right)$ for the initial datum $u_0$ in $L^2_r\left((R,+\infty)\right)$ and boundary value $g$ in $L^2\left((0,T)\right).$ This solution satisfies the energy estimate \eqref{aprioriestimate}.

\paragraph{Regular solutions.}
To solve the mixed problem in Sobolev spaces we need some compatibility conditions. For instance, the initial and the boundary conditions imply that necessarily
 \begin{equation}
\mathbf{e}_2\cdot u_{0_{|_{r=R}}}=g_{|_{t=0}}=\mathbf{e}_2\cdot u_{|_{t=0,r=R}},
\end{equation}
if the traces are defined. Let us consider the generic equation $$\partial_t u =f-A(\overline{u})\partial_r u -B(\overline{u},r)u.$$ 
Then, we can formally (this is the meaning of the brackets \textquotedblleft\textquotedblright) define
$$``\partial_t u_{|_{t=0}}" = -A(\overline{u}_0)_{|_{t=0}}\partial_r u_0 -B(\overline{u_0},r)_{|_{t=0}}u_0 +f_{|_{t=0}}.$$ 
Hence, provided traces are defined, 
\begin{equation*}
\begin{aligned}	\mathbf{e}_2\cdot {``\partial_t u_{|_{t=0}}"}_{|_{r=R}}=\mathbf{e}_2\cdot (f_0-A(\overline{u}_0)\partial_r u_0 -B(\overline{u}_0,r)u_0)_{|_{r=R}}=g_1
\end{aligned}\end{equation*}where $g_1:=\partial_t g_{|_{t=0}}$.
These conditions are necessary for the existence of a smooth solution. We can continue the expansion to higher orders looking for more compatibility conditions.
Let us introduce the following notation (as in \cite{Met2} and \cite{Schochet}):
\begin{notation}Let us write the linear equation in \eqref{linearU} 
	$$\partial_t u=F(\overline{u},\partial)u + f$$
	with $$F(\overline{u},\partial)=-A(\overline{u})\partial_r -B(\overline{u},r).$$ 
We define formally the traces $u_j:=``\partial^j_t u_{|_{t=0}}"$ as functions of $u_0$ determined inductively by
  \begin{equation}
  	u_0=u_{|_{t=0}} \ \ \ \ \ \ u_{j+1}= \overline{\mathcal{F}}_{j}(u_0,...u_{j})  +  f_{j}
  	\label{defuj}
  \end{equation}
	with \begin{equation}\overline{\mathcal{F}}_{j}(u_0,...u_{j})= \sum_{p + |k|\leq j}^{}A_{j,p,k} (\overline{u}_0) \overline{u}_{(k)}\partial_r u_p + \sum_{p + |k|\leq j}^{}B_{j,p,k} (\overline{u}_0,r) \overline{u}_{(k)}u_p  \label{barF}\end{equation} where we use the notation
	$$\mbox{for} \ \  \ \  k=(k_1,\ . \ . \ . \ , k_r), \ \ \ \ \ \ \ \ \ \ u_{(k)}=u_{k_1}\ . \ . \ . \ u_{k_r}.$$
\end{notation} 
 Note that $u_j$ is not the derivative of a known function but rather the value that the derivative of $u$ will have if $u$ exists. Therefore necessarily smooth enough solutions to \eqref{linearU} must satisfy
\begin{equation}
	\mathbf{e}_2 \cdot {u_j}_{|_{r=R}}=g_j
	\label{compatibility}
\end{equation}with $g_j:=\partial^j_t g_{|_{t=0}}.$
\begin{definition}
The data $u_0\in H^k_r(\mathbb{R_+})$, $g\in H^k((0,T))$ and $f\in H^k_r((0,T)\times\mathbb{R_+})$ satisfy the compatibility conditions up to order $s\leq k-1$ if \eqref{compatibility} holds for each $j=0,...,s$.
\end{definition}
 For the linear floating structure mixed problem \eqref{linearU} the compatibility conditions \eqref{compatibility} on the initial data $u_0=(\zeta_{e,0}, q_{e,0})$ and the boundary value $g=q_{i_{|_{r=R}}}$ can be written as:
\begin{equation}\begin{aligned}
 {q_e}_{|_{r=R,t=0}}=q_{i_{|_{r=R,t=0}}},  \ \ \ \ \ \ \ \ 
 {\overline{\mathcal{F}}_{j}}_{2}(\zeta_{e,0},...,\zeta_{e,j},q_{e,0},...,q_{e,j})_{|_{r=R}}=\partial_t^j q_{i_{|_{r=R,t=0}}}, \ \ \  \ \ \ \ j\geq1.
\end{aligned}
\label{compacondiFS}
\end{equation} where ${\overline{\mathcal{F}}_{j}}_{2}$ is the second component of ${\overline{\mathcal{F}}_{j}}.$
As in the $L^2$ case, our goal is to find an a priori estimate for the linear problem \eqref{linearU} in order to get existence and uniqueness of the solution in some more regular space.
\begin{proposition}\label{estimateX^k}
Let $T>0$ and $k\geq 1$ be an integer. Assume that $\overline{u} \in X^s(T)$ with $s=\max(k,2)$ satisfies \eqref{subsonic}. With $\alpha$ as in Lemma \ref{p1p2} there are a constant $c_{\alpha,R}$ and non-decreasing functions $C_R(\cdot), C_{1,k,R}(\cdot)$ and $C_{2,k,R}(\cdot)$ on $[0,+\infty)$ such that all the solutions $u \in H_r^{k+1}((0,T)\times (R,+\infty))$ solving \eqref{linearU} satisfy:
\begin{equation}
\begin{aligned}
\Xnorm{u(t)}^2 +& \normsob{u_{|_{r=R}}}{(0,t)}^2 \leq c_{\alpha,R}e^{tC_{\alpha,R,k}(\|\overline{u}\|_{X^s(T)})} \times \\&\times \left(\Xnorm{u_0}^2+\normsob{g}{(0,t)}^2+ K_{k,R}(\|\overline{u}\|_{X^s(T)})\int_{0}^t\|f(\tau)\|^2_{X^k}d\tau\right)
\end{aligned}
\label{aprioriestimateHs}
\end{equation}for all $t\in [0,T]$ with $$C_{\alpha,R,k}(\|\overline{u}\|_{X^s(t)})=1+ \alpha^{-1}C_R(\|\overline{u}\|_{X^2(t)})+ \alpha^{-1}(k+1)C_{1,k,R}(\|\overline{u}\|_{X^s(t)})$$ and
$$K_{k,R}(\|\overline{u}\|_{X^s(T)})=\dfrac{C_{2,k,R}(\|\overline{u}\|_{X^s(T)})}{\max(c(\|\overline{u}(0)\|_{X^2}),c_2 R)}.$$
\end{proposition}
\begin{proof}We adapt here the argument presented in \cite{Met}.
We denote by $u^i$ the tangential derivative $\partial_t^i u$ for $i\leq k$, which in the one dimensional case is simply the time derivative, and we introduce the tangential norm
$$\Xnorm{u(t)}^\prime:=\sum_{i=0}^{k}\normL{\partial_t^i u(t)}{(R,+\infty)}$$ We apply $\partial_t^i$ to the equation of \eqref{linearU} and we get
\begin{equation}
	\partial_t u^i + A(\overline{u})\partial_r u^i + B(\overline{u},r)u^i= [A(\overline{u}),\partial_t^i]\partial_r u + [B(\overline{u},r), \partial_t^i] u + f^i,
	\label{D^iequation}
\end{equation} $$\mathbf{e}_2\cdot u^i_{|_{r=R}}=g^i.$$
As we have done in the previous $L^2$ case we consider $\dfrac{d}{dt}(S(\overline{u})u^i,u^i)_{L^2_r((R,+\infty))}$. The only difference from the previous case is the presence of the two commutator terms in \eqref{D^iequation}. We need to control their $L^2_r$ norms in a different way.\\
The first term can be written under the form $\sum\limits_{\alpha=1,...,i}\|\partial_t^\alpha(A(\overline{u}))\partial_t^{i-\alpha}\partial_r u \|_{L^2_r}$. For $\alpha\leq k-1$ every term of the sum is controlled by
\begin{equation}
	\begin{aligned}\|\partial_t^\alpha (A(\overline{u}))\|_{L^\infty}\| \partial_t^{i-\alpha}\partial_r u \|_{L^2_r}&\leq c_R\|\partial_t^\alpha (A(\overline{u}))\|_{H^{k-\alpha}_r}\| u\|_{X^{k-\alpha+1}}\\&\leq c_R \|A(\overline{u}) \|_{X^k}\|u\|_{X^k}\end{aligned}
	\label{Adrucontrol}\end{equation}using the fact that $1\leq\alpha\leq k-1$ for the Sobolev embedding and that \linebreak$\|\cdot\|_{X^\lambda}\leq\|\cdot\|_{X^\delta}$ if $\lambda < \delta.$ Recall that Sobolev embeddings $H^k_r\hookrightarrow W^{s,\infty}$ still hold for the weighted spaces $H^k_r$ since we are considering the exterior domain $(R,+\infty)$. For $\alpha=k$ we directly have 
$$\|\partial_t^k (A(\overline{u}))\|_{L^2_r}\|\partial_r u\|_{L^\infty}\leq c_R \|A(\overline{u})\|_{X^k}\|u\|_{X^k}$$ since $u\in H^{k+1}_r$ with $k\geq1.$ We can find the same estimate for the commutator term with $B(\overline{u},r)$.\\
We recall the following Moser-type estimate for the $\|\cdot\|_{X^k}$ norm:
\begin{lemma}[Schochet]
For $A(\cdot)$ smooth enough, the following holds 
	\begin{equation}
		\|A(\overline{u})\|_{X^k}\leq C_k (1+\|\overline{u}\|^k_{X^k})
		\label{MoserSchochet}
	\end{equation}with $k\geq 1.$
\end{lemma} We refer to the Appendix B of \cite{Schochet} for the details of the proof based on Gagliardo-Nirenberg inequalities. Better estimates can be used, in particular if one wants to derive blow up conditions as Métivier in \cite{Met2}, but here we are not interested in this problem. Hence we get the following inequality: 
  \begin{equation}
  \begin{aligned}
  &\left(S(\overline{u}) u^i(t),u^i(t)\right)_{L^2((R,+\infty))} \leq \left(S(\overline{u}(0)) u^i_0,u^i_0\right)_{L^2((R,+\infty))}\\&+\int_{0}^{t}\left[\left(S(\overline{u}) f^i(\tau),f^i(\tau)\right)_{L^2((R,+\infty))}+c_2R\left|{g^i(\tau)}\right|^2-c_1R\left|u^i_{|_{r=R}}(\tau)\right|^2\right]d\tau\\&+\left(1+\alpha^{-1}C_R(\|\overline{u}\|_{X^2(T)})\right)\int_{0}^{t}(S(\overline{u})u^i(\tau),u^i(\tau))_{L^2((R,+\infty))}d\tau\\&+ C_{k,R}(\|\overline{u}\|_{X^s(T)})\int_0^t\|u(\tau)\|^2_{X^k}d\tau
  \end{aligned}
  \label{ineqprodS}\end{equation}with $s=\max(k,2)$ and some non-decreasing function $C_{k,R}(\cdot)$ on $[0,+\infty)$. Here $u^i_0=u_i$ where the $u_i$ are defined in \eqref{defuj}. We note that for solutions to \eqref{linearU}, we have $$\Xnorm{u(0)}^\prime=\sum_{j=0}^{k}\|u_j\|_{L^2_r}, \ \ \ \  \ \ \ \ \ \ \Xnorm{u(0)}=\sum_{j=0}^{k}\|u_j\|_{H^{k-j}_r}.$$ 
 For $u$ satisfying the equation \eqref{linearU}, we have 
 \begin{equation*}
 	\partial_r u =A^{-1}(\overline{u})(f -\partial_t u -B(\overline{u},r)u)
 \end{equation*}and we can show that the $X^k$ norm is controlled by the tangential one. The following holds for $0\leq t\leq T$:
  \begin{equation}
  \begin{aligned}\|u(t)\|^2_{X^k}&\leq C_1(\|\overline{u}\|_{X^s(T)}){\|u(t)\|^\prime_{X^k}}^2 + C_2(\|\overline{u}\|_{X^s(T)})\|f(t)\|^2_{X^k}\\&\leq C_1(\|\overline{u}\|_{X^s(T)})\alpha^{-1}\sum_{i=0}^{k}\left(S(\overline{u}) u^i,u^i\right)_{L^2((R,+\infty))}+ C_2(\|\overline{u}\|_{X^s(T)})\|f(t)\|^2_{X^k} \end{aligned}
  \label{tangentialestimate}
  \end{equation} with some non-decreasing functions $C_1(\cdot), C_2(\cdot)$ on $[0,+\infty)$. In the second inequality we have used the fact that $S(\overline{u})\geq \alpha \text{Id}.$ By taking the sum for $i$ from $0$ to $k$ and by applying Gronwall's Lemma we obtain
   \begin{equation}
   \begin{aligned}&\alpha
   {\Xnorm{u(t)}^\prime}^2 +c_1R\|u_{|_{r=R}}\|_{H^k((0,t))}^2\leq \ e^{t\left(1+ \alpha^{-1}C_R(\|\overline{u}\|_{X^2(T)})+ \alpha^{-1}(k+1)C_{1,k,R}(\|\overline{u}\|_{X^s(T)})\right)}\times\\&\times\big( c(\|\overline{u}(0)\|_{X^2}){\Xnorm{u(0)}^\prime}^2+c_2R\left\|g\right\|_{H^k((0,t))}^2+ C_{2,k,R}(\|\overline{u}\|_{X^s(T)})\int_{0}^{t}{\|f(\tau)\|_{X^k}}^2d\tau\big)
   \end{aligned}\label{tang} 
   \end{equation}
   with some non-decreasing functions $C_{1,k,R}(\cdot), C_{1,k,R}(\cdot)$ on $[0,+\infty)$.
   By definition of the tangential norm we have ${\Xnorm{u_0}^\prime}^2\leq {\Xnorm{u_0}}^2 $. We use \eqref{tangentialestimate} and the estimate \eqref{aprioriestimateHs} follows with $$c_{\alpha,R}=\dfrac{\max(c(\|\overline{u}(0)\|_{X^2}),c_2 R)}{\min(\alpha,c_1R)},$$  
    $$C_{\alpha,R,k}(\|\overline{u}\|_{X^s(T)})=1+ \alpha^{-1}C_R(\|\overline{u}\|_{X^2(T)})+ \alpha^{-1}(k+1)C_{1,k,R}(\|\overline{u}\|_{X^s(T)})$$
    and $$K_{k,R}(\|\overline{u}\|_{X^s(T)})=\dfrac{C_{2,k,R}(\|\overline{u}\|_{X^s(T)})}{\max(c(\|\overline{u}(0)\|_{X^2}),c_2 R)}.$$
\end{proof}
 Equivalently to the $L^2$-case, we can state the following theorem:
\begin{theorem}\label{linearex}
	Let $k\geq 1$ be an integer and $T>0$. Suppose $u_0\in H^k_r((R,+\infty))$,\linebreak $g\in H^k((0,T))$ and $f\in  H^k_r((0,T)\times (R,+\infty))$ satisfy the compatibility conditions \eqref{compatibility} up to the order $k-1$. Assume that $\overline{u}\in X^s(T)$ with $s=\max(k,2)$ satisfies \eqref{subsonic}. Then, there is a unique solution $u\in X^k(T)$ to \eqref{linearU}. Its trace on $r=R$ belongs to $H^k((0,T))$ and $u$ satisfies the energy estimate \eqref{aprioriestimateHs}.
\end{theorem}
\begin{proof}
We show only the idea of the proof of the existence. For more details and the proof of uniqueness see \cite{Met2}. First we solve the equation with a loss of smoothness. We consider the data $u_0$, $f$ and $g$ in $H^{k+2}_r$ satisfying the compatibility conditions up to order $k$. One can prove that there is a solution in $H^{k+1}_r((0,T)\times (R,+\infty))\subset X^{k}(T)$, by extending the data by 0 for $t<0$ and then by applying the existence result for the mixed problem in $(-\infty,T]\times (R,+\infty)$ of \cite{Met}.\\The second step is to consider $H^k$-data: we use the compatibility conditions up to order $k-1$ to approximate $u_0$, $f$ and $g$ in $H^k_r$ and $H^k$ with sequences \linebreak$u_0^n \in H^{k+2}_r((R,+\infty))$, $f^n\in H^{k+2}_r((0,T)\times(R,+\infty))$ and $g^n\in H^{k+2}((0,T))$ satisfying the compatibility conditions up to order $k+1$. From the previous argument and the energy estimate \eqref{aprioriestimateHs} we have that $u^n$ is a Cauchy sequence in $X^k(T)$ and therefore converges to the limit $u\in X^k(T)$, which is a solution to \eqref{linearU} since $k\geq 1$.
\end{proof}

\subsection{The quasilinear problem and application to the case of a solid with prescribed motion}
In the particular case of the floating structure problem, the boundary condition in \eqref{quasilinearU} is $g=q_{i_{|_{r=R}}}$, the value of the horizontal discharge in the interior domain at the boundary $r=R$. We will see in the next section that this quantity is strictly linked to the solid motion, in particular to the vertical component of the velocity of the center of mass $w_G(t)$.\\ 
In the case of a solid with prescribed motion, the boundary condition $g$ is a datum of the problem. Hence, after having studied the linear problem \eqref{linearU}, one can use a standard iterative scheme argument in order to get existence and uniqueness of the solution to \eqref{quasilinearU}. 
\begin{theorem}\label{prescribedmotion}
	Consider a solid with a prescribed motion. For $k\geq 2$, let $u_0 \in H_r^k((R,+\infty))$ and $w_G\in H^{k}((0,T))$ satisfy the compatibility conditions in Definition \ref{compcondFS} up to order $k-1$. Assume that $u_0$ satisfies \eqref{subsonic}. Then, the fluid problem \eqref{couplingFluid} with boundary condition $-\frac{R}{2}w_G(t)$ admits a unique solution $u\in X^k(T)$ with $X^k(T)$ as in \eqref{Xk}.	
\end{theorem}
\begin{proof}[Sketch of the proof]
We introduce the iterative scheme by defining the sequence $(u^n)_n$ with $u^{n}$ solution to the linear problem $L(u^{n-1})u^{n}=0.$ 
The existence of such a sequence is given by Theorem \ref{linearex}. Once we have showed the control of the sequence in some \textquotedblleft big norm\textquotedblright  and the convergence in some \textquotedblleft small norm\textquotedblright , the limit $u$ of $(u^n)_n$ is the solution to \eqref{quasilinearU}. For more details we refer to \cite{Met2}. We will show a detailed proof in the case of a free motion in Theorem \ref{maintheorem} below. 
\end{proof}
From this point on we consider a solid with free motion. Therefore the boundary condition is still an unknown of the problem and we must adapt the classical argument used in Theorem \ref{prescribedmotion} to our problem introducing an iterative scheme for the fluid-structure coupled system.
The details of this coupled iterative scheme argument are given in Section 5. Before, we deal with the solid problem and we deduce an ordinary differential equation describing the motion of its center of mass.

\section{The solid equation}
In this section we address the motion of the solid. We recall that we are considering a floating structure moving only vertically.\\ Denoting $m$ the mass of the body, $g$ the gravity acceleration and $z_G$ the vertical position of the center of mass, we consider only the vertical component of Newton's law for the conservation of linear momentum:
\begin{equation}
m \ddot{z}_G = -mg + F_{\mathrm{fluid}}
\label{newton}
\end{equation} where $F_{\mathrm{fluid}}=2\pi\int_{0}^{R}(\underline{P}_i-P_{\mathrm{atm}})rdr$ is the resulting vertical force exerted by the fluid on the solid.\\
Let us introduce the displacement $\delta_G(t) := z_G(t) - z_{G,eq}$ between the vertical position of the center of mass at time $t$ and at its equilibrium. In the case of vertical motion $h_w(t,r)=h_{w,eq}(r) + \delta_G(t)$, where $h_{w,eq}$ is the fluid height at the equilibrium.
For simplicity we consider a cylindrically symmetric solid with flat bottom, which means that the wetted surface $\zeta_w$ (hence $h_w$) does not depend on the spatial coordinate in the interior domain $(0,R)$. See Appendix A for the general case with a cylindrically symmetric solid with a non-flat bottom.\\
\begin{proposition}\label{solideq}
Newton's law \eqref{newton} can be written under the following form:
\begin{equation}
(m+m_a(\delta_G)) \ddot{\delta}_G(t)= -\mathfrak{c} \delta_G(t) + \mathfrak{c} \zeta_e (t,R) + \left(\dfrac{\mathfrak{b}}{h_e^2(t,R)}+\beta(\delta_G)\right)\dot{\delta}_G^2(t)
\label{eqmotion}
\end{equation} with $$ \mathfrak{c}=\rho g \pi R^2, \ \ \ \ \ \ \mathfrak{b}=\dfrac{\pi \rho R^4}{8},$$
$$m_a(\delta_G)=\dfrac{\mathfrak{b}}{ h_w(\delta_G)}=\dfrac{\mathfrak{b}}{ h_{w,eq}+\delta_G(t)},$$ $$\beta(\delta_G)=\dfrac{\mathfrak{b}}{2h^2_w(\delta_G)}=\dfrac{\mathfrak{b}}{2(h_{w,eq}+\delta_G(t))^2}.$$
\end{proposition}

\begin{remark}
In \eqref{eqmotion} $m_a(\delta_G)$ is called the \textit{added mass} term and it represents the fact that, in order to move in the fluid, the solid has to accelerate itself but also the portion of fluid next to it. This effect appears in other hydrodynamical configurations, in particular for totally submerged solids studied by Glass, Sueur and Takahashi \cite{SueurTaka} and Glass, Munnier and Sueur \cite{VortexP}. It has an important role for the stability of numerical simulations of fluid-structure interactions \cite{NumAddesMass}.\\
The coupling with the fluid motion is given by the term $\zeta_e(t,R)$ and $\dfrac{\mathfrak{b}}{h_e^2(t,R)}$, which means that the solid motion depends on the value of the elevation of the exterior free surface at the boundary $r=R$.
\label{remarkcoup}
\end{remark}

\begin{proof}
We decompose $F_{\mathrm{fluid}}$ according to the decomposition \eqref{P1} - \eqref{P4} of the pressure 
$$F_{\mathrm{fluid}}=F^\mathrm{I}_{\mathrm{fluid}}+F^\mathrm{II}_{\mathrm{fluid}}+F^\mathrm{III}_{\mathrm{fluid}}$$ with $$F^\mathrm{I}_{\mathrm{fluid}}=2\pi\int_{0}^R(\underline{P}_i^\mathrm{I}-P_{\mathrm{atm}})rdr, \ \ \ \ \ \ \ \ \  F^\mathrm{j}_{\mathrm{fluid}}=2\pi\int_{0}^{R}\underline{P}_i^\mathrm{j}rdr, \ \ \ j=\mathrm{II},\mathrm{III}.$$
Using the elementary potential $\Phi_\mathcal{I}^r$ defined in \cite{Lan} we can write 
\begin{equation*}F^\mathrm{II}_{\mathrm{fluid}}=-2\pi\int_{0}^{R}\underline{P}_i^\mathrm{II}\left(\partial_r +\frac{1}{r}\right)\left(h_w\partial_r \Phi_\mathcal{I}^r\right)rdr,\end{equation*}and, after integration by parts,
\begin{equation*}
\begin{aligned}
F^\mathrm{II}_{\mathrm{fluid}}&=-2\pi\int_{0}^{R}\left(\partial_r +\frac{1}{r}\right)\left(h_w\partial_r \underline{P}_i^\mathrm{II}\right)\Phi_\mathcal{I}^rrdr\\&=-2\pi\rho\int_{0}^R\dot{w}_G\Phi_\mathcal{I}^rrdr,\end{aligned}\end{equation*}
where the second equality comes from the definition \eqref{P2} of $\underline{P}_i^\mathrm{II}$. Using again the definition of elementary potential we obtain
\begin{equation*}
	\begin{aligned}
	F^\mathrm{II}_{\mathrm{fluid}}&=2\pi\rho\left[\int_{0}^{R}\Phi_\mathcal{I}^r\left(\partial_r +\frac{1}{r}\right)\left(h_w\Phi_\mathcal{I}^r\right)\right]\dot{w}_Grdr\\
	&=-2\pi\rho\left[\int_{0}^{R}\frac{1}{h_w}\left(h_w\partial_r\Phi_\mathcal{I}^r\right)^2\right]\dot{w}_Grdr.
	\end{aligned}
\end{equation*}From the definition of the elementary potential we explicitly have that
\begin{equation*}
	h_w\partial_r\Phi_\mathcal{I}^r=-\frac{r}{2}.
\end{equation*}
It follows that
\begin{equation*}
	F^\mathrm{II}_{\mathrm{fluid}}=-m_a(h_w)\dot{w}_G,
\end{equation*} 
with $m_a(h_w)$ as in \eqref{eqmotion}. Proceeding similarly we can write also 
 \begin{equation*}
F^\mathrm{\mathrm{I}}_{\mathrm{fluid}}=-2\pi\rho\int_{0}^{R}\dfrac{r}{2h_w}\dfrac{h_w}{\rho}\partial_r\underline{P}^\mathrm{I}_irdr.\end{equation*} Then, \eqref{newton} becomes
\begin{equation}
\left(m+m_a(h_w)\right)\dot{w}_G= -mg + F^\mathrm{I}_{\mathrm{fluid}}+2\pi\int_{0}^R\underline{P}_i^\mathrm{III}rdr.
\label{addedmasseffect}
\end{equation}
Moreover,  \eqref{P1} can be written as $$\partial_r y(r) + \dfrac{y(r)}{r}=-\dfrac{b(r)}{r} \ \ \ \ \ \ \mbox{in} \ \ \ (0,R)$$ with $$y(r)=\dfrac{h_w}{\rho}\partial_r \underline{P}_i^\mathrm{I}+\partial_r\left(\dfrac{q_i^2}{h_w}\right) \ \  \ \ \ \mbox{and} \ \ \ \ \   \  \ \ b(r)=\partial_r \left(\dfrac{q_i^2}{h_w}\right).$$ Hence we have \begin{equation}y(r)=-\dfrac{1}{r}\left(\dfrac{q_i(r)^2}{h_w}-\dfrac{q_i(0)^2}{h_w(0)}\right).
\label{ry}\end{equation}
Because of the constraint \eqref{interiorconstraint}, the mass conservation equation of \eqref{WWFloatpolar} in the interior domain becomes
$$\partial_r q_i + \dfrac{1}{r}q_i = -\dot{\delta}_G(t) \ \ \ \ \mbox{in} \ \ \ (0,R)$$ then we have \begin{equation}
q_i(t,r)=-\dfrac{r}{2}\dot{\delta}_G(t).
\label{coupboundvalue}
\end{equation} 
Hence $q_i(t,0)=0$ and \eqref{ry} becomes $$\dfrac{h_w}{\rho}\partial_r \underline{P}_i^\mathrm{I}=-\partial_r\left(\dfrac{q_i^2}{h_w}\right)-\dfrac{1}{r}\dfrac{q_i^2}{h_w}=-\dfrac{3}{4h_w}r\dot{\delta}_G^2.$$ Replacing the expression of $\dfrac{h_w}{\rho}\partial_r \underline{P}_i^\mathrm{I}$ in $F^\mathrm{I}_{\mathrm{fluid}}$ we get $$F^\mathrm{I}_{\mathrm{fluid}}=\dfrac{3\pi \rho R^4}{16h^2_w}\dot{\delta}_G^2$$
and, by definition of the equilibrium state, we have 
\begin{equation}-mg -2\pi\rho g\int_{0}^{R}\zeta_{w,eq}rdr=0.\label{defeq}\end{equation} Since the solid has vertical side-walls the following equality holds
\begin{equation}
2\pi\rho g \int_{0}^R\zeta_w(t)rdr-2\pi\rho g \int_{0}^R\zeta_{w,eq}rdr=\mathfrak{c}\delta_G(t).\label{diffvol}\end{equation} These two equalities, together with the constraint $\zeta_i=\zeta_w$, give $$-mg=\mathfrak{c} \zeta_i (t,R)-\mathfrak{c} \delta_G(t).$$
Solving the elliptic problem \eqref{P4} whose solution is the constant (in space) boundary value, we obtain the nonlinear second order ordinary differential equation \eqref{eqmotion}.
\end{proof}
\begin{remark}
All the computations in the proof of Proposition \ref{solideq} reduce to particular cases of the Green's identity:
\begin{equation}
	\int_0^R p \text{div}_r (h_w\partial_r q) rdr= \int_{r=R} h_w p \partial_r q r - \int_{r=R} h_w q \partial_r p r+ \int_{0}^R q \text{div}_r(h_w \partial_r p)rdr
\end{equation} with particular $p$ and $q=\Phi^r_\mathcal{I}$,  where $div_r=\partial_r + \frac{1}{r}$ is the divergence in the axisymmetric configuration.
\end{remark}
 Recall that in \eqref{coupboundvalue} we have 
\begin{equation*}
	q_i(t,R)= -\frac{R}{2}\dot{\delta}_G(t). 
\end{equation*}This term is the boundary value in the fluid mixed problem \eqref{WWFloatpolarext}. It follows that this is the coupling term between the fluid and the solid motion in the fluid system, as $\zeta_e(t,R)$ has the same property in the solid equation (see Remark \ref{remarkcoup}).

\section{Fluid-structure coupling}
From the previous two sections, it follows that the fluid-structure interaction problem considered in this paper is described by the following mathematical model:
\begin{proposition}
The nonlinear shallow water equations with a floating structure for an axisymmetric flow without swirl take the form 
\begin{equation}
\begin{cases}
\partial_t u +A(u)\partial_r u +B(u,r)u=0 \ \ \ \ \mbox{in} \ \ \ (R,+\infty)\\
\mathbf{e}_2 \cdot u_{|_{r=R}}=-\dfrac{R}{2}\dot{\delta}_G(t)\\[5pt]
u(0)=u_0,
\end{cases}
\label{couplingFluid}
\end{equation}
with $A(u), B(u,r)$ as in \eqref{quasilinearU}. Moreover the solid motion is given by the Cauchy problem
\begin{equation}
\begin{cases}
(m+m_a(\delta_G)) \ddot{\delta}_G(t)= -\mathfrak{c} \delta_G(t) + \mathfrak{c} {\mathbf{e}_1 \cdot u}_{|_{r=R}}  +\left(\mathfrak{b}(u)+\beta(\delta_G)\right)\dot{\delta}_G^2(t),\\
\delta_G(0)=\delta_0,\\
\dot{\delta}_G(0)=\delta_1,
\end{cases}
\label{couplingSolid}
\end{equation}with $m_a(\delta_G)$, $\beta(\delta_G)$ as in \eqref{eqmotion},
$$\mathfrak{b}(u)=\dfrac{\mathfrak{b}}{\left(\mathbf{e}_1 \cdot u_{|_{r=R}}+h_0\right)^2} =\dfrac{\mathfrak{b}}{h_e(t,R)^2},$$using the fact that $$\zeta_e(t,R)={\mathbf{e}_1 \cdot u}_{|_{r=R}} \ \ \ \ \ \ \mbox{and} \ \ \ \ \ \ h_e(t,R)= h_0 + \zeta_e(t,R).$$
\end{proposition}
 Let us give the notion of compatibility conditions in the case of this particular fluid-structure coupled problem. We recall the equation in \eqref{couplingFluid}
$$\partial_t u=F(u,\partial)u$$
with \begin{equation}F(u,\partial)=A(u)\partial_r + B(u,r)\label{F(u)}\end{equation}
We define formally the traces $u_j:=``\partial^j_t u_{|_{t=0}}"$ as functions of $u_0$ determined inductively by
\begin{equation}
u_0=u_{|_{t=0}} \ \ \ \ \ \ u_{j+1}=\mathcal{F}_{j}(u_0,...u_{j}) 
\label{defujnonlin}
\end{equation}
with \begin{equation*}\mathcal{F}_{j}(u_0,...u_{j})= \sum_{p + |k|\leq j}^{}A_{j,p,k} (u_0) u_{(k)}\partial_r u_p + \sum_{p + |k|\leq j}^{}B_{j,p,k} (u_0,r) u_{(k)}u_p  .\end{equation*} where we use the notation
$$\mbox{for} \ \  \ \  k=(k_1,\ . \ . \ . \ , k_r), \ \ \ \ \ \ \ \ \ \ u_{(k)}=u_{k_1}\ . \ . \ . \  u_{k_r}.$$

\begin{definition}\label{compcondFS}
The data $u_0\in H_r^k((R,+\infty))$, $\delta_0\in\mathbb{R}$ and $\delta_1\in\mathbb{R}$ of the floating structure coupled system \eqref{couplingFluid} - \eqref{couplingSolid} satisfy the compatibility conditions up to order $ k-1$ if, for $0\leq j\leq k-1$, the following holds:
\begin{equation*}
\begin{aligned}
 \mathbf{e}_2 \cdot {u_j}_{|_{r=R}}= -\frac{R}{2}\delta_{j+1},
\end{aligned}
\end{equation*}where $\delta_{j+1}$ are the formal traces $^{``}\dfrac{d^{j+1}}{{dt}^{j+1}}\delta_G (0)^"$ defined from the ODE in \eqref{couplingSolid} as
$$\frac{d^{j-1}}{dt^{j-1}}\left[\dfrac{1}{\left(m+m_a(\delta_G)\right)}\left(-\mathfrak{c}\delta_G
+ \mathfrak{c}\mathbf{e}_1 \cdot u_{|_{r=R}}+
 \left(\mathfrak{b}(u)+\beta(\delta_G)\right)\dot{\delta}_G^2\right)\right](0). $$
\end{definition}
In the following theorem we prove that the coupled model \eqref{couplingFluid} - \eqref{couplingSolid} is locally in time well-posed:
\begin{theorem}
	For $k\geq 2$, let $u_0=(\zeta_{e,0}, q_{e,0}) \in H_r^k((R,+\infty)),$ $ \delta_0$ and $\delta_1$ satisfy the compatibility conditions in Definition \ref{compcondFS} up to order $k-1$. Assume that there exist some constants $h_\mathrm{min}, c_\mathrm{sub}>0$ such that 
	$$\forall r\in(R,+\infty): \qquad h_{e,0}(r)\geq h_\mathrm{min}, \quad \left(gh_{e,0} - \frac{q^2_{e,0}}{h^2_{e,0}}\right)(r)\geq c_\mathrm{sub}, $$with $h_{e,0}=h_0 +\zeta_{e,0},$ and that $$\delta_0 > -h_{w,eq}$$ with the constant $h_{w,eq}$ as in Section 4. Then, the coupled problem \eqref{couplingFluid} - \eqref{couplingSolid} admits a unique solution $(u,\delta_G)\in X^k(T)\times H^{k+1}((0,T))$ with $X^k(T)$ as in \eqref{Xk}.
	\label{maintheorem}
\end{theorem}
\begin{remark}
	Considering an initial datum $u_0\in H^2_r\left((R,+\infty)\right)$, we need the following compatibility conditions satisfied:
	\begin{equation*}q_e(0,R)= -\frac{R}{2}\delta_1\end{equation*} and 
	\begin{equation*}
	\begin{aligned}&-\partial_r \left(\frac{q^2_e}{h_e}\right)(0,R)-\frac{1}{R}\frac{q^2_e}{h_e}(0,R)-gh_e(0,R)\partial_r\zeta_e(0,R)\\&=
	-\dfrac{R}{2\left(m+m_a(\delta_0)\right)}\left(-\mathfrak{c}\delta_0 + \mathfrak{c}\zeta_e(0,R)+\left(\frac{\mathfrak{b}}{h_e^2(0,R)} +\beta(\delta_0)\right)\delta_1^2\right).\end{aligned}\end{equation*}
	For instance let us take the initial configuration of the fluid-structure interaction as the following: the solid displaced from its equilibrium position with no initial velocity, which means $$\delta_0\neq0 \ \  \mbox{and}\ \ \delta_1=0, $$ and the fluid such that$$ h_e(0,R)=h_0,  \ \ q_e(0,R)=0, \ \ \partial_r \zeta_e(0,R)=-\frac{\mathfrak{c}\delta_0 R}{2\left(m+m_a(\delta_0)\right)gh_0}.$$
	Then, the initial conditions are compatible and we can apply Theorem \ref{maintheorem}.
\end{remark}

\begin{proof} We adapt here the argument that Métivier used in \cite{Met2} for the existence and uniqueness of the solution to the fluid mixed problem and we couple it with the solid equation. Similar techniques are used in \cite{IguLan} by Iguchi and Lannes.\\
	We introduce the following iterative scheme for the coupled system \eqref{couplingFluid} - \eqref{couplingSolid}. 	For $u_0\in H^k_r((R,+\infty))$, let us consider the linear mixed problem
	\begin{equation}
	\begin{cases}
	\partial_t u^{n} +A(u^{n-1})\partial_r u^n +B(u^{n-1},r)u^n=0, \ \ \ \ \mbox{in} \ \ \ (R,+\infty)\\
	\mathbf{e}_2 \cdot u^n_{|_{r=R}}=-\dfrac{R}{2}\dot{\delta}^{n-1}_G(t)\\
	u^n(0)=u_0.
	\end{cases}
	\label{couplingFluidscheme}
	\end{equation}
	and the linear ODE
	\begin{equation}
	\begin{cases}
	(m+m_a(\delta_G^{n-1})) \ddot{\delta}^n_G= -\mathfrak{c} \delta_G^n + \mathfrak{c} \mathbf{e}_1 \cdot u^{n}_{|_{r=R}} +\left(\mathfrak{b}(u^{n-1})+\beta(\delta^{n-1}_G)\right){\dot{\delta}_G}^{n-1}{\dot{\delta}_G}^{n},\\
	\delta^n_G(0)=\delta_0,\\
	\dot{\delta}^n_G(0)=\delta_1,
	\end{cases}
	\label{couplingSolidscheme}
	\end{equation}
	Our goal is to find the solution of the coupled system as the limit of the previous iterative scheme. Hence we need to show the existence and the convergence of the sequence $V^n=(u^n,\delta_G^n)$. 
	We consider the product space $X^k(T) \times H^{k+1}((0,T))$ endowed with the norm 
	$$\|V^n\|_{coup,k}:=\|u^n\|_{X^k(T)} + \|\delta_G^n\|_{H^{k+1}((0,T))}$$
	with the $X^k(T)$ norm defined as in \eqref{Xk}. We denote by $E$ the subspace $$E=\{V=(u, \delta_G)\in X^k(T) \times H^{k+1}((0,T))\,|\, \|V\|_{coup,k}\leq \widetilde{R}\},$$  for some $\widetilde{R}>0$ to determine later,  such that 
	\begin{equation*}
	\forall t\in[0,T), \forall r\in(R,+\infty):\qquad h_e(t,r)\geq C_0,\quad
	\left(gh_e - \frac{q^2_e}{h^2_e}\right)(t,r)\geq c_0\end{equation*}and \begin{equation*} \|\delta_G - \delta_0\|_{L^\infty((0,T))}\leq M_0
	\end{equation*}for some constants $0<C_0\leq h_{\mathrm{min}} $, $0<c_0\leq c_{\mathrm{sub}}$ and $M_0= \frac{\delta_0 + h_{w,eq}}{2}>0$. We choose the first element of the sequence $V^0=(u^0,\delta_G^0)$ with $u^0\in H^{k+\frac{1}{2}}_r(\mathbb{R}\times(R,+\infty))$ such that $$\partial_t^j u^0_{|_{t=0}}=u_j,\ \ \ 0\leq j \leq k$$ with $u_j$ as in \eqref{defujnonlin}.
	We can assume that $u^0$ vanishes for $|t|\geq1$, hence $u^0\in X^{k}(T)$ for all $T.$  There exists a constant $K_0=K_0(u_0, \delta_0, \delta_1)$ depending only on the data such that  
	\begin{equation}
	\|u^0\|_{X^k(T)} + \|\delta_G^0\|_{H^k((0,T))}\leq K_0.
	\label{K0}
	\end{equation} We have that $V^0\in E$ choosing $\widetilde{R}\geq K_0$ .  
	We suppose that $V^{n-1}=(u^{n-1}, \delta_G^{n-1})$ is constructed in $E\subseteq X^{k}(T)\times H^{k+1}((0,T))$ for some $T>0$ with 
	\begin{equation}
	\partial_t^j u^{n-1}_{|_{t=0}}=u_j, \ \ \ j\leq k.
	\label{inductivehyp}
	\end{equation} 
	For $n=1$ this is true. By the definition \eqref{barF} of $\overline{\mathcal{F}}_j$  and by \eqref{inductivehyp}, $$\partial_t^j (F(u^{n-1},\partial)u^n)_{|_{t=0}}=\overline{\mathcal{F}}_j(u_0,...,u_j)$$with $\overline{u}=u^{n-1}.$ Now we consider the linear problem \eqref{couplingFluidscheme}. We compute $u_j^n$ using \eqref{defuj}. We can see that $u_j^n=u_j$ with the $u_j$ defined before. Then, the compatibility conditions $\mathbf{e}_2\cdot {u_j}_{|_{r=R}}=-\frac{R}{2}\delta_{j+1}$ imply that the data $ \dot{\delta}^{n-1}_G$ and $u_0$ are compatible for the linear problem. From Theorem \ref{linearex} the system \eqref{couplingFluidscheme} has a unique solution $u^n\in X^{k}(T)$ and $\partial_t^j u^n_{|_{t=0}}=u_j^n=u_j.$ Moreover,
	$$\|u^n(0)\|_{X^k}=\sum_{j\leq k}\|u_j\|_{H^k_r}\leq K_0.$$
	Therefore we can continue the construction and this permits to define a sequence $u^n\in X^{k}(T)$ solving the linear problem \eqref{couplingFluidscheme} and, from \eqref{aprioriestimateHs}, satisfying 
	\begin{equation}
	\begin{aligned}
	\|{u^n}\|_{X^k(T)} + \normsob{u^n_{|_{r=R}}}{(0,T)} \leq  C(K_0)e^{TC(\|u^{n-1}\|_{X^{k}(T)})}\left(K_0+\normsob{\dot{\delta}_G^{n-1}}{(0,T)}\right).  
	\end{aligned}
	\label{estimateXku_n}
	\end{equation}
	The existence and uniqueness of $\delta_G^n\in W^{2,\infty}((0,T))$ is given by the Cauchy-Lipschitz-Picard theorem since the coefficients in \eqref{couplingSolidscheme} are bounded when $V^{n-1}=(u^{n-1}, \delta_G^{n-1})\in E$. 
	We want to show that $V^n=(u^n, \delta_G^n)\in E$. To do that, we now provide a control of product estimates in Sobolev spaces in time; of course one has the standard estimate $$\| f g\|_{H^k((0,T))}\leq C(T) \| f \|_{H^k((0,T))} \| g\|_{H^k((0,T))}$$ but the constant $C(T)$ blows up as $T\to 0$ which raises some issues since we are led to choose $T$ small enough in the proof. We therefore use the following more precise lemma where the time dependence of the constants is made explicit (see Proposition \ref{proofLemma} in Appendix B for the proof):  
	\begin{lemma}\label{prodest}
		Let $k\geq 1$ be an integer. For $f,g\in H^k((0,T))$ the following holds:
		\begin{equation}
		\begin{aligned}
		\|fg\|_{H^k((0,T))}\lesssim \ \  & \ \ \sqrt{T}\|f\|_{H^k((0,T))}\|g\|_{H^k((0,T))}\\&+ (|f(0)|+|\frac{d}{dt}f(0)|+ ...+ |\frac{d^{k-1}}{dt^{k-1}}f(0)|)\|g\|_{H^k((0,T))}\\&  + (|g(0)|+|\frac{d}{dt}g(0)|+ ...+ |
		\frac{d^{k-1}}{dt^{k-1}}g(0)|)\|f\|_{H^k((0,T))}  
		\end{aligned}
		\end{equation}
	\end{lemma} 
	The previous lemma yields the following estimate for the solution $\delta^n_G$ to \eqref{couplingSolidscheme}:
	\begin{proposition}
		The solution $\delta_G^n$ to \eqref{couplingSolidscheme} satisfies
		\begin{equation}
		\begin{aligned}
		&\left\|\delta_G^n\right\|_{H^{k+1}((0,T))}\leq \alpha(T,\widetilde{R}) + \beta(T,\widetilde{R})\|u^n_{|_{r=R}}\|_{H^k((0,T))} 
		\end{aligned}
		\label{estimateDeltan}
		\end{equation}with $\alpha(T,\widetilde{R}),\beta(T, \widetilde{R})\rightarrow 0$ as $T\rightarrow 0.$
	\end{proposition}
	\begin{proof}
		It is immediate to derive
		\begin{equation}
		\left\|\delta_G^n\right\|_{H^{k+1}((0,T))}\leq C_1(T) + C_2(T)\|\ddot{\delta}_G^n\|_{H^{k-1}((0,T))}
		\label{controldelta}
		\end{equation} with $C_1(T), C_2(T) \rightarrow 0$ as $T\rightarrow 0$.
		Using the equation on $\ddot{\delta}^n_G$ we can estimate  $\ddot{\delta}^n_G$ in the following way:
		\begin{equation*}
		\begin{aligned}
		\|\ddot{\delta}^n_G\|_{H^{k-1}}&\lesssim \sqrt{T}\left\|\dfrac{\mathfrak{c}}{m+m_a(\delta_G^{n-1})}
		\right\|_{H^{k-1}}\left(\|\delta_G^n\|_{H^{k-1}}+\|\mathbf{e}_1 \cdot u^n_{|_{r=R}}\|_{H^{k-1}}\right) \\
		&\ \ \ +\sqrt{T}\left\|\dfrac{\mathfrak{b}(u^{n-1})+\beta(\delta_G^{n-1})}{m+m_a(\delta_G^{n-1})}\dot{\delta}_G^{n-1}
		\right\|_{H^{k-1}}\|\dot{\delta}_G^n\|_{H^{k-1}}\\
		& \ \ \ +C_0\left\|\dfrac{\mathfrak{c}}{m+m_a(\delta_G^{n-1})}
		\right\|_{H^{k-1}}+
		D_0\left(\|\delta_G^n\|_{H^{k-1}}+\|\mathbf{e}_1 \cdot u^n_{|_{r=R}}\|_{H^{k-1}} \right)\\
		&  \ \ \ +C_1\left\|\dfrac{\mathfrak{b}(u^{n-1})+\beta(\delta_G^{n-1})}{m+m_a(\delta_G^{n-1})}\dot{\delta}_G^{n-1}
		\right\|_{H^{k-1}}+ D_1 \|\dot{\delta}_G^n\|_{H^{k-1}}
		%
		%
		\end{aligned}
		\end{equation*}with $$C_0=C_0\left(\left|\delta^n_0\right|,..., \left|\delta^n_{k-2}\right|, \left|\mathbf{e}_1 \cdot u^n_{|_{r=R}}(0)\right|,..., \left|\frac{d^{k-2}}{dt^{k-2}}\mathbf{e}_1 \cdot u^n_{|_{r=R}}(0)\right|\right),$$
		$$D_0=D_0\left(\left|\dfrac{\mathfrak{c}}{m+m_a(\delta_G^{n-1})}(0)\right|,...,\left|\frac{d^{k-2}}{dt^{k-2}} \dfrac{\mathfrak{c}}{m+m_a(\delta_G^{n-1})}(0)\right|\right),$$
		$$C_1=C_1\left(\left|\delta^n_1\right|,..., \left|\delta^n_{k-1}\right|\right),$$
		$$D_1=D_1\left(\left|\dfrac{\mathfrak{b}(u^{n-1})+\beta(\delta_G^{n-1})}{m+m_a(\delta_G^{n-1})}\dot{\delta}_G^{n-1}(0)\right|,...,\left|\frac{d^{k-2}}{dt^{k-2}} \dfrac{\mathfrak{b}(u^{n-1})+\beta(\delta_G^{n-1})}{m+m_a(\delta_G^{n-1})}\dot{\delta}_G^{n-1}(0)\right|\right).$$ 
		By applying Lemma \ref{prodest} several times to the products in the previous estimate and using the fact that for $V^{n-1}=(u^{n-1}, \delta_G^{n-1})\in E$ the denominators are bounded from below, we get 
		\begin{equation}
		\begin{aligned}
		\|\ddot{\delta}^n_G\|_{H^{k-1}}\lesssim& \,C(T, \|\delta_G^{n-1}\|_{H^{k+1}}, \|u^{n-1}\|_{H^k})\|\delta_G^n\|_{H^{k+1}}
		+ C(T, \|\delta_G^{n-1}\|_{H^{k+1}})\| u^n_{|_{r=R}}\|_{H^k} 
		\end{aligned}
		\label{estddotdelta}
		\end{equation}	Here the constants may not tend to zero as $T$ goes to zero but they are bounded. Then, using the control for $\|V^{n-1}\|_{coup,k}$, for $T$ small enough we can move the first term in the right hand side of \eqref{estddotdelta} to the left of the inequality \eqref{controldelta} and \eqref{estimateDeltan} follows.
	\end{proof}
	From \eqref{estimateXku_n} and \eqref{estimateDeltan}, we get the following estimate for the coupled norm: 
	\begin{equation*}
	\begin{aligned}
	\|V^n \|_{coup,k}&\leq \|u^n\|_{X^k(T)} + \|\delta_G^n\|_{H^{k+1}((0,T))}\\
	&\leq \|u^n\|_{X^k(T)}+ \alpha(T,\widetilde{R}) + \beta(T,\widetilde{R})\|u^n_{|_{r=R}}\|_{H^k((0,T))} \\
	&\leq C(T,\widetilde{R}, K_0, \|u^{n-1}\|_{X^k(T)}) ( K_0 + \|\delta_G^{n-1}\|_{H^{k+1}((0,T))})+ \alpha(T,\widetilde{R}).
	\end{aligned}
	\end{equation*} 
	Using again the control for $\|V^{n-1}\|_{coup,k}$, we can find some $\widetilde{R}\geq K_0$ such that for $T$ small enough 
	\begin{equation*}
	\|V^n\|_{coup,k}\leq \widetilde{R}.
	\end{equation*} For $u_0=(\zeta_{e,0}, q_{e,0})$  we have
	$$	h_e^{n}(t)=h_{e,0} + \int_{0}^{t}\partial_t h_e^{n}$$ with 
	$$\left|\int_{0}^{t}\partial_t h_e^{n}\right|\leq T\|\partial_t  h_e^{n}\|_{L^\infty((0,T))L^\infty_r((R,+\infty))}\leq T\|u^n\|_{X^k(T)}\leq T\widetilde{R}.$$
	Moreover, 
	$$gh_e^n - \left(\tfrac{q_e^n}{h_e^n}\right)^2  = gh_{e,0} - \left(\tfrac{q_{e,0} }{h_{e,0} }\right)^2 + \int_{0}^{t}\partial_t (gh_e^n - \left(\tfrac{q_e^n}{h_e^n}\right)^2)$$ with 
	\begin{align*}\left| \int_{0}^{t}\partial_t \left(gh_e^n - \left(\tfrac{q_e^n}{h_e^n}\right)^2\right)\right| &\leq T \left\|\partial_t \left(gh_e^n - \left(\tfrac{q_e^n}{h_e^n}\right)^2\right)\right\|_{L^\infty((0,T))L^\infty_r((R,+\infty))}\\&\leq T \|A(u^n)\|_{X^k(T)}\leq TC(1 + \widetilde{R}^k)\end{align*}
	where in the last inequality we have used \eqref{MoserSchochet}.
	Finally,
	$$\|\delta^n_G -\delta_0\|_{L^\infty((0,T))}\leq \sqrt{T}\|\delta_G^n\|_{H^{k+1}((0,T))}\leq \sqrt{T}\widetilde{R}.$$
	Hence, using the assumption on the initial data, the time existence $T$ can be shorten to get 
	\begin{equation*}
	\forall t\in[0,T), \forall r\in(R,+\infty)\qquad h_e^n(t,r)\geq ,\quad
	\left(gh^n_e - \left(\tfrac{q^n_e}{h^n_e}\right)^2\right)(t,r)\geq c_0\end{equation*}and \begin{equation*} \|\delta^n_G - \delta_0\|_{L^\infty((0,T))}\leq M_0.
	\end{equation*}
	Now we look for the convergence of the sequence $V^n$ in a \textquotedblleft smaller\textquotedblright \,norm. We consider the space $$X^0(T)\times H^1((0,T))=C^0([0,T], L^2_r((R,+\infty)))\times H^1((0,T)).$$
	We have that $u^{n}-u^{n-1}$ satisfies 
	\begin{equation*}\begin{cases}
	\partial_t \left(u^n-u^{n-1}\right) +A(u^{n-1})\partial_r \left(u^n - u^{n-1}\right) +B(u^{n-1},r)\left(u^n-u^{n-1}\right)\\\qquad=-\left(A(u^{n-1})-A(u^{n-2})\right)\partial_r u^{n-1} - \left(B(u^{n-1},r)-B(u^{n-2},r)\right)u^{n-1},\\[5pt]
	\mathbf{e}_2 \cdot\left(u^n-u^{n-1}\right)_{|_{r=R}}=-\dfrac{R}{2}\left(\dot{\delta}^{n-1}_G(t)-\dot{\delta}^{n-2}_G(t)\right),\\
	\left(u^n-u^{n-1}\right)(0)=0.
	\end{cases}\end{equation*}Using the embedding $H^k_r \hookrightarrow W^{1,\infty}$ for $k\geq 2$  it yields
	\begin{equation}
	\begin{aligned}
	&\bigg\|\left(A(u^{n-1})-A(u^{n-2})\right)\partial_r u^{n-1} + \left(B(u^{n-1},r)-B(u^{n-2},r)\right)u^{n-1}\bigg\|^2_{X^0}\\&\qquad
	\leq K(\|u^{n-1}\|_{X^k(T)}) \|u^{n-1}-u^{n-2}\|^2_{X^0}.
	\end{aligned}
	\end{equation}Then, by \eqref{aprioriestimate} it follows
	\begin{equation}
	\begin{aligned}
	&\|(u^n - u^{n-1})\|^2_{X^0(T)}+ \|(u^n-u^{n-1})_{|_{r=R}}\|_{L^2((0,T))}^2\leq C(K_0) e^{TC(\|u^{n-1}\|_{X^2(T)})}) \times \\ &\times \left(\|\dot{\delta}^{n-1}_G-\dot{\delta}^{n-2}_G\|_{L^2((0,t))}^2 + K(\|u^{n-1}\|_{X^k(T)})\int_{0}^{T} \|(u^{n-1}-u^{n-2})(\tau)\|^2_{X^0}d\tau\right).
	\end{aligned}
	\label{estimatediffun}
	\end{equation} On the other hand, we have 
	\begin{equation}
	\begin{aligned}
	\|\delta_G^n -\delta_G^{n-1}\|_{H^1((0,T))}^2\leq C_2(T)\|\ddot{\delta}_G^n -\ddot{\delta}_G^{n-1}\|_{L^2((0,T))}^2, 
	\end{aligned}
	\end{equation} with $C_2(T)\rightarrow 0$ as $T\rightarrow 0$.
	Since in the ODE \eqref{couplingSolid} the terms 
	$$\dfrac{\mathfrak{c}}{m+m_a(\delta_G)}, \ \ \ \ \dfrac{\mathfrak{b}(u)}{m+m_a(\delta_G)}, \ \ \ \ \dfrac{\beta(\delta_G)}{m+m_a(\delta_G)}$$
	are Lipschitz continuous on $(u,\delta_G)\in E$ from $L^2$ to $L^2$ and considering the equation for $\delta_G^n$ and $\delta_G^{n-1}$, we obtain the following estimate for $T$ small enough:
	\begin{equation}
	\begin{aligned}
	\|\delta_G^n -\delta_G^{n-1}\|_{H^1((0,T))}&\leq   \widetilde{\alpha}(T,\widetilde{R})\| (u^n-u^{n-1})_{|_{r=R}}\|_{L^2((0,T))}\\& \qquad + \widetilde{\beta}(T,\widetilde{R}) \|\delta_G^{n-1}-\delta_G^{n-2}\|_{L^2((0,T))}
	\end{aligned}
	\label{estimatediffdeltan}
	\end{equation}for some constants $\widetilde{\alpha}(T,\widetilde{R}), \widetilde{\beta}(T,\widetilde{R}).$ Therefore, using \eqref{estimatediffun} and \eqref{estimatediffdeltan}, we get 
	\begin{equation*}
	\begin{aligned}
	&	\|V^n-V^{n-1}\|_{coup,0}=\|u^n-u^{n-1}\|_{X^0(T)} + \|\delta_G^n-\delta_G^{n-1}\|_{H^1((0,T))}\\
	&\leq \|u^n-u^{n-1}\|_{X^0(T)} +\widetilde{\alpha}(T,\widetilde{R})\| (u^n-u^{n-1})_{|_{r=R}}\|_{L^2((0,T))}\\&\qquad +\widetilde{\beta}(T,\widetilde{R}) \|\delta_G^{n-1}-\delta_G^{n-2}\|_{L^2((0,T))}\\
	&\leq C(T,K_0,\widetilde{R})\left(\|\delta_G^{n-1}-\delta_G^{n-2}\|_{H^1((0,T))} +\int_0^T\|(u^{n-1}-u^{n-2})(t)\|_{X^0}dt\right) \\
	&\leq K(T,K_0,\widetilde{R}) \|V^{n-1}-V^{n-2}\|_{coup,0}.
	\end{aligned}
	\end{equation*}where we have used $\|u^{n-1}\|_{X^2(T)}\leq \|u^{n-1}\|_{X^k(T)}$ and the inductive hypothesis \eqref{induction}. Then, we can choose $T$ small enough such that $K(T,K_0,\widetilde{R})<1$ and we obtain that $V^n$ is a convergent sequence in $X^0(T)\times H^1((0,T))$ with limit $V=(u,\delta_G)$. By standard arguments (see \cite{Schochet}) we have that $V\in E\subseteq X^k(T)\times H^{k+1}((0,T))$ is the unique solution of the coupled problem \eqref{couplingFluid} - \eqref{couplingSolid}.
	
\end{proof}
\appendix
\section{Solid with non-flat bottom}
We derive here the equation for the solid motion, as in Section 4, in the more general case of a solid with non-flat bottom. Due to the fact that the interior and exterior domains do not change during the motion, we suppose that the contact between the free surface and the floating structure takes place on the vertical side-walls during all the motion. Then, we can state the following proposition:
\begin{proposition}
	In the case of a solid with non-flat bottom, Newton's law \eqref{newton} can be written under the following form:
	\begin{equation}
	(m+m_a^{NF}(\delta_G)) \ddot{\delta}_G(t)= -\mathfrak{c} \delta_G(t) + \mathfrak{c} \zeta_e (t,R) +\left(\frac{\mathfrak{b}}{h_e^2(t,R)}+\beta^{NF}(\delta_G)\right)\dot{\delta}_G^2(t)
	\label{eqmotionnonflat}
	\end{equation} with $$\mathfrak{c}=\rho g \pi R^2, 	\ \ \ \ \ \ \mathfrak{b}=\dfrac{\pi \rho R^4}{8}, \ \ \ \ \ \  m_a^{NF}(\delta_G)=\dfrac{\rho \pi}{2}\int_{0}^{R}\dfrac{r^3}{h_w(\delta_G,r)}dr,
	$$  $$\beta^{NF}(\delta_G)=\dfrac{\mathfrak{b}}{2 h_w^2(\delta_G,R)}+\dfrac{\pi \rho}{8}\int_{0}^{R}\dfrac{r^4}{h_w^3(\delta_G,r)}\partial_r h_w(\delta_G,r) \ dr,
	$$ and the dependence on $\delta_G$ given by $$h_w(\delta_G,r)=\delta_G(t)+h_{w,eq}(r).$$
\end{proposition}
\begin{remark}
	One can note that in the case of a solid with flat bottom $\left(\partial_r h_w(\delta_G,r)=0\right)$
	$$m_a^{NF}(\delta_G)=m_a(\delta_G), \ \ \ \ \ \ \beta^{NF}(\delta_G)=\beta(\delta_G),$$ and \eqref{eqmotionnonflat} coincides with \eqref{eqmotion}.	 
\end{remark}
\begin{proof}
	We derive only the expression of $F^\mathrm{I}_{\mathrm{fluid}}$ and $F^\mathrm{III}_{\mathrm{fluid}}$ in the case of a solid with a non-flat bottom. The added mass term comes from the fact that $F^\mathrm{II}_{\mathrm{fluid}}$ can be written as
	$$F^\mathrm{II}_{\mathrm{fluid}}=-m_a^{NF}(h_w)\dot{w}_G$$ with 
	$$m_a^{NF}(h_w)=\dfrac{\rho \pi}{2}\int_{0}^{R}\dfrac{r^3}{h_w}dr.$$
	By definition, $$F^\mathrm{I}_{\mathrm{fluid}}=2\pi\rho\int_{0}^{R}\dfrac{r}{2h_w}\left(-\dfrac{h_w}{\rho}\partial_r\underline{P}_i^\mathrm{I}\right)rdr$$ with $\underline{P}_i^\mathrm{I}$ defined as the solution to \eqref{P1}. Since we want $$\underline{P}^\mathrm{I}_i-P_{\mathrm{atm}} \in H^1_{0,r}((0,R))$$ we get 
	$$-\dfrac{h_w}{\rho}\partial_r\underline{P}_i^\mathrm{I}=\partial_r \left(\dfrac{q_i^2}{h_w}\right)+\dfrac{q_i^2}{rh_w}+gh_w\partial_r h_w.$$ Using the formula for the horizontal discharge in the interior domain \linebreak $q_i(t,r)=-\dfrac{r}{2}\dot{\delta}_G(t)$, we obtain that
	$$F^\mathrm{I}_{\mathrm{fluid}}=\pi\rho \int_{0}^{R}\left(\dfrac{3r^3}{4h_w^2}\dot{\delta}_G^2 + (gr^2 -\dfrac{r^4}{4h_w^3}\dot{\delta}_G^2)\partial_r h_w \right)dr.$$
	Also in the case of a solid with non-flat bottom, \eqref{P4} admits the unique constant solution$$\underline{P}_i^\mathrm{III}(t,r)=\rho g (\zeta_e(t,R)-\zeta_i(t,R))+ \frac{\rho}{2}q^2_i(t,R)\left( \frac{1}{h_e^2(t,R)}-\frac{1}{h_w^2(t,R)}\right),$$ with $q_i(t,R)=-\frac{R}{2}\dot{\delta}_G(t).$ By definition of $F^\mathrm{III}_{\mathrm{fluid}}$ we have $$F^\mathrm{III}_{\mathrm{fluid}}=\mathfrak{c}(\zeta_e(t,R)-\zeta_i(t,R)) +\mathfrak{b}\left(\frac{1}{h_e^2(t,R)}-\frac{1}{h_w^2(t,R)}\right)\dot{\delta}^2_G(t).$$ The relations \eqref{defeq} and \eqref{diffvol} still hold but in this case we obtain
	$$-mg=2\pi\rho g \int_{0}^{R}r\zeta_i(t,r)dr - \mathfrak{c}\delta_G(t).$$ 
	Now we observe that the term $\pi \rho \int_{0}^{R}gr^2\partial_r h_w dr$ can be written by integration by parts as 
	$$\pi \rho \int_{0}^{R}gr^2\partial_r h_w dr= \mathfrak{c}\zeta_i(t,R)- 2\pi\rho g  \int_{0}^{R}r\zeta_i(t,r)dr$$ where $\partial_r h_w=\partial_r \zeta_w =\partial_r \zeta_i$ using the constraint \eqref{interiorconstraint}. Putting all these expressions in Newton's law \eqref{newton} and integrating by parts, we get \eqref{eqmotionnonflat}.
\end{proof}

\section{Proof of Lemma \ref{prodest}}
We prove here the product estimate \eqref{productestimate}:
\begin{proposition}	\label{proofLemma}
	Let $k\geq 1$ be an integer. For $f,g\in H^k((0,T))$ the following estimate holds:
	\begin{equation}
	\begin{aligned}
	\|fg\|_{H^k((0,T))}\lesssim \ \  & \ \sqrt{T}\|f\|_{H^k((0,T))}\|g\|_{H^k((0,T))}\\&+ (|f(0)|+|\frac{d}{dt}f(0)|+ ...+ |\frac{d^{k-1}}{dt^{k-1}}f(0)|)\|g\|_{H^k((0,T))}\\&  + (|g(0)|+|\frac{d}{dt}g(0)|+ ...+ |
	\frac{d^{k-1}}{dt^{k-1}}g(0)|)\|f\|_{H^k((0,T))}  
	\end{aligned}
	\label{productestimate}
	\end{equation}
\end{proposition}
\begin{proof}
	We write $f(t)$ as \begin{equation*}
	f(t)=f(0) + \int_0^t\frac{d}{dt}f(s)ds, 
	\end{equation*}hence
	\begin{equation}
	\|f\|_{L^\infty((0,T))}\leq |f(0)| + \sqrt{T}\left\|f\right\|_{H^1((0,T))}
	\label{inje}
	\end{equation}
	We prove \eqref{productestimate} by induction. For $k=1$ we have 
	\begin{equation*}
	\begin{aligned}
	&\|fg\|_{H^1((0,T))}\lesssim 	\|fg\|_{L^2((0,T))} + 	\|\frac{df}{dt} g\|_{L^2((0,T))} + 	\|f\frac{dg}{dt} \|_{L^2((0,T))}\\&
	\leq (\|f\|_{L^2((0,T))} + 	\|\frac{df}{dt}\|_{L^2((0,T))})	\|g\|_{L^\infty((0,T))} + \|f\|_{L^\infty((0,T))}\|g\|_{H^1((0,T))}\\&
	\leq \sqrt{2}\|f\|_{H^1((0,T))}	\|g\|_{L^\infty((0,T))} + \|f\|_{L^\infty((0,T))}\|g\|_{H^1((0,T))}
	\end{aligned}
	\end{equation*} 	
	and using \eqref{inje} we get 
	\begin{equation}
	\begin{aligned}
	\|fg\|_{H^1((0,T))} \lesssim & \ \sqrt{T}\|f\|_{H^1((0,T))} \|g\|_{H^1((0,T))}\\& + |f(0)|\|g\|_{H^1((0,T))} + 	|g(0)|\|f\|_{H^1((0,T))} 
	\end{aligned}
	\end{equation}
	Let us suppose that \eqref{productestimate} is true for $k-1$. Then, we have
	\begin{equation*}
	\begin{aligned}
	&\|fg\|_{H^k((0,T))}\leq \|fg\|_{H^{k-1}((0,T))} + \|\frac{d^k}{dt^k}(fg)\|_{L^2((0,T))}\\
	&\lesssim \|fg\|_{H^{k-1}((0,T))} + \|\frac{d^k}{dt^k}f\|_{L^2((0,T))}\|g\|_{L^\infty((0,T))} + \|f\|_{L^\infty((0,T))}\|\frac{d^k}{dt^k}g\|_{L^2((0,T))}\\&\  + \|\frac{d}{dt}f\|_{L^2((0,T))}\|\frac{d^{k-1}}{dt^{k-1}}g\|_{L^\infty((0,T))} + \sum_{i=2}^{k-1}C_{k,i} \|\frac{d^i}{dt^{i}}f\|_{L^\infty((0,T))} \|\frac{d^{k-i}}{dt^{k-i}}f\|_{L^2((0,T))}.
	\end{aligned}
	\end{equation*}
	From the estimate \eqref{inje} for $f,g, \dfrac{d^{k-1}}{dt^{k-1}}g$ and $\dfrac{d^i}{dt^{i}}f$ we get
	\begin{equation*}
	\begin{aligned}
	\|fg\|_{H^k((0,T))}\lesssim & \  \|fg\|_{H^{k-1}((0,T))} + 3\sqrt{T} \|f\|_{H^k((0,T))}\|g\|_{H^k((0,T))}\\&
	+ (|f(0)|+|\frac{d^2}{dt^2}f(0)|+ ...+ |\frac{d^{k-1}}{dt^{k-1}}f(0)|)\|g\|_{H^k((0,T))}\\&  + (|g(0)|+|
	\frac{d^{k-1}}{dt^{k-1}}g(0)|)\|f\|_{H^k((0,T))}  
	\end{aligned}
	\end{equation*}
	and \eqref{productestimate} follows using the inductive hypothesis.\end{proof}

\newpage
\bibliographystyle{siam}
\bibliography{bibliography}

\end{document}